\documentclass[preprint,12pt]{imsart}
\RequirePackage[T1]{fontenc}
\usepackage{color}
\usepackage[latin1]{inputenc}
\usepackage[T1]{fontenc}
\RequirePackage{amsfonts,amsthm,amsmath}
\RequirePackage[colorlinks,citecolor=blue,urlcolor=blue]{hyperref}
\usepackage{a4wide}
\newtheorem{thm}{Theorem}
\newtheorem{prop}{Proposition}

\newtheorem{lemme}{Lemma}
\newtheorem{rmq}{Remark}
\newtheorem{exemple}{Example}

\usepackage{amsmath}
\usepackage{amsfonts}
\usepackage{enumerate}
\usepackage{multirow}
\usepackage{graphicx}
\usepackage[square,sort,comma,numbers,authoryear]{natbib}
\usepackage{float}
\startlocaldefs
\theoremstyle{plain}
\endlocaldefs

\usepackage{sansmath}
\sansmath

\hypersetup{pdfstartview=FitH}

\ifpdf
    \graphicspath{{PDFFIG/}}

\else
    \graphicspath{{EPSFIG/}}
\fi

\allowdisplaybreaks[1]

\begin{document}
\begin{frontmatter}
\title{Fast calibration of weak FARIMA models\\[0.5cm]
}
\runtitle{ One-step estimation of weak FARIMA models}
\begin{aug}
\author{\fnms{\Large{Samir}}
\snm{\Large{Ben Hariz}$^{\hspace*{0.07cm} \boldsymbol{1}}$}
\ead[label=e1]{Samir.Ben$\_$Hariz@univ-lemans.fr}}\\[2mm]
\author{\fnms{\Large{Alexandre}}
\snm{\Large{Brouste}$^{\hspace*{0.07cm} \boldsymbol{1}}$}
\ead[label=e2]{Alexandre.Brouste@univ-lemans.fr}}\\[2mm]
\author{\fnms{\Large{Youssef}}
\snm{\Large{Esstafa}$^{\hspace*{0.07cm} \boldsymbol{1}}$}
\ead[label=e3]{Youssef.Esstafa@univ-lemans.fr}}\\[2mm]
\author{\fnms{\Large{Marius}}
\snm{\Large{Soltane}$^{\hspace*{0.07cm} \boldsymbol{2}}$}
\ead[label=e4]{marius.soltane@utc.fr}}\\[2mm]
\runauthor{S. Ben Hariz, A. Brouste, Y. Esstafa and M. Soltane}
\begin{minipage}{7cm}
\affiliation{Le Mans Universit\'e}
\address{\hspace*{0cm}\\
$^{\hspace*{0.07cm} \boldsymbol{1}}$Le Mans Universit\'e,\\
Institut du Risque et de l'Assurance,\\
Laboratoire Manceau de Math\'ematiques,\\
Avenue Olivier Messiaen,\\
72085 Le Mans Cedex 9, France.\\[0.2cm]}
\end{minipage}
\begin{minipage}{7cm}
\affiliation{Universit\'e Bourgogne Franche-Comt\'e}
\address{\hspace*{0cm}\\
$^{\hspace*{0.07cm} \boldsymbol{2}}$LMAC EA 2222\\
Universit\'e de Technologie de Compi\`egne\\
CS 60319 - 57 Avenue de Landshut\\
60203 Compi\`egne Cedex, France.\\[0.2cm]}
\end{minipage} \\

\vspace{0.5cm}

\footnotesize\sf\href{Samir.Ben$\_$Hariz@univ-lemans.fr}{Samir.Ben$\_$Hariz@univ-lemans.fr}\\
\footnotesize\sf\href{Alexandre.Brouste@univ-lemans.fr}{Alexandre.Brouste@univ-lemans.fr}\\
\footnotesize\sf\href{Youssef.Esstafa@univ-lemans.fr}{Youssef.Esstafa@univ-lemans.fr}\\
\footnotesize\sf\href{marius.soltane@utc.fr}{marius.soltane@utc.fr}
\end{aug}

\vspace{0.5cm}
\begin{abstract} In this paper, we investigate the asymptotic properties of Le Cam's one-step estimator for weak Fractionally AutoRegressive Integrated Moving-Average (FARIMA) models. For these models, noises are uncorrelated but neither necessarily independent nor martingale differences errors. We show under some regularity assumptions that the one-step estimator is strongly consistent and asymptotically normal with the same asymptotic variance as the least squares estimator. We show through simulations that the proposed estimator reduces computational time compared with the least squares estimator. An application for providing remotely computed indicators for time series is proposed.
\end{abstract}
\begin{keyword}[class=AMS]
\kwd[Primary ]{62M10, 62M15}
\kwd[; Secondary ]{91B84}
\end{keyword}
\begin{keyword} Weak FARIMA models; Le Cam's one-step estimator; Least squares estimator; Consistency; Asymptotic Normality
\end{keyword}
\end{frontmatter}

\section{Introduction}
Time series often exhibit linear and/or nonlinear dependence. When large correlations for small lags are detected, short memory processes can suffice to model the dependence structure of the series. The ARMA processes (see for example \cite{Box&Jenkins} and \cite{fz98}) or VARMA for the multivariate framework (see \cite{lutkepohl} and \cite{yac-these}) are examples of short memory processes.

However, in many scientific disciplines and many applied fields, including hydrology, climatology, economics, finance and computer science, the autocorrelations of some time series decrease very slowly. This phenomenon may be due to several factors, in particular nonstationarity and/or long-range dependence.

A large audience of mathematicians has been attracted by long memory processes and their various applications (see for instance \cite{M1965}, \cite{MandelbrotVan} and \cite{Mandelbrot-Wallis1, Mandelbrot-Wallis2, Mandelbrot-Wallis3}), Granger in economics (\cite{Granger&Joyeux}), Dobrushin in physics (\cite{dobrushin1979non}) and, even earlier, Hurst in hydrology (\cite{hurst1951}). Long-range dependent processes constitute currently one of the popular areas of statistical research and occupy a central place in the time series literature (see \cite{Granger&Joyeux}, \cite{hosking1981LR}, \cite{Fox&Taqqu1986}, \cite{Dahlhaus1989}, \cite{palma}, \cite{Beran2013}, \cite{BMES2019}, among others).

The fractional autoregressive integrated moving-average (FARIMA, for short) model is widely used to model the long memory phenomenon. This model was first introduced by \cite{Granger&Joyeux} and then generalized to take into account short-term fluctuations in time series by \cite{hosking1981LR}. FARIMA models therefore have the advantage of jointly modeling the long memory behavior of time series and their short-term dynamics through a fractional integration parameter $d$ and autoregressive and moving-average parameters respectively. Their fame is partly due to their structure similar to the one of standard ARIMA models in which the differentiation exponent $d$ is an integer. FARIMA models are generally used with strong assumptions on the noise that limit their generality. We call strong FARIMA the standard models in which the error term is assumed to be an independent and identically distributed sequence (iid for short), and we speak about weak FARIMA models when the errors are uncorrelated but neither necessarily independent nor martingale differences. It is common in the time series literature to talk about the subclass of semi-strong FARIMA models when the associated innovation process is a semi-strong white noise, that is a stationary martingale difference.  An example of semi-strong white noise is the generalized autoregressive conditional heteroscedastic (GARCH) model (see \cite{FZ2010}). The distinction between strong, semi-strong or weak FARIMA models is therefore only a matter of noise assumptions with the following inclusions:
$$\left\{\text{strong FARIMA} \right\}\subset\left\{\text{semi-strong FARIMA}
\right\}\subset\left\{\text{weak FARIMA}\right\}.$$

The independence of the noise in strong FARIMA models is often considered to be very restrictive for many time series with general nonlinear dependencies\footnote{See for instance \cite{Tong1990}, \cite{fz98}, \cite{frz}, \cite{BLR2006}, \cite{FY2008}, \cite{FZ2010}, \cite{yac}, \cite{yac2}, \cite{shao2011}, \cite{bmcf}, \cite{shaox}, \cite{BM2014} and \cite{BMS2018} for some references on nonlinear time series models.}. Weak FARIMA models correct this problem by allowing the noise to contain very general nonlinear dependencies of often unidentified structures. They therefore have the great advantage of providing linear modeling to nonlinear processes.

The asymptotic theory of estimation is mainly limited to strong and semi-strong FARIMA models. Whittle's estimator (see  \cite{Whittle1953}) is commonly used to estimate the parameters of FARIMA models (see for example \cite{Fox&Taqqu1986}, \cite{Dahlhaus1989}, \cite{GS1990} and \cite{Taqqu1997}).
The study of the asymptotic properties of this estimator is developed in the case where the errors are assumed to be independent and identically distributed and in the framework where the noise is considered to be a martingale difference (see \cite{Beran1995}, \cite{baillie1996}, \cite{ling-li}, \cite{hauser1998}, \cite{palma}, \cite{Beran2013}, among others). All this work is limited to the case where the nonlinear dependency is absent or with a well-identified structure. For example, in financial time series modeling, in order to capture conditional heteroscedasticity, it is common that innovations in FARIMA models are assumed to have a GARCH structure (see for example \cite{baillie1996}, \cite{hauser1998}).\\
For weak FARIMA models, the asymptotic normality has been obtained for Whittle's estimator (\cite{shaox, shao2010}) and the LSE (\cite{BMES2019}). In this paper, we propose a new calibration of weak FARIMA models based on Le Cam's one-step approach (see \cite{Le-Cam}). In Le Cam's one-step procedure, an initial guess estimator is corrected by a single step of Newton gradient descent method on loglikelihood function. We adapt in this work the Le Cam one-step procedure. Firstly, we propose the LSE on subsample as an initial estimator. Secondly, since we do not specify the distribution of the weak white noise, the single Newton step is done on the least squares functional.  This estimator greatly reduces the computation time and preserves the same asymptotic properties as the LSE. One-step procedure has shown its efficiency in terms of computation time and precision for diffusion processes \cite{KM15, GY20}, ergodic Markov chains \cite{KM16}, fractional Gaussian noise observed at high frequency \cite{BSV20} or stable noise~\cite{BM18} and inhomogeneous Poisson processes \cite{DGK18}. 

The paper is organized as follows. In Section~\ref{sec:one-step}, we introduce the model and the notations used in the sequel and we give the asymptotic properties of Le Cam's one-step estimator of the parameters of weak FARIMA models. In Section~\ref{sec:sim}, we provide some numerical illustrations to show the performance of the proposed estimator on finite sample sizes. All the technical proofs are gathered in Section~\ref{sec:proofs}.

\section{Le Cam's one-step estimation of weak FARIMA models}\label{sec:one-step}
In this section we present the parametrization that is used in the sequel and we study the asymptotic properties of the Le Cam one-step estimator of long memory FARIMA processes induced by uncorrelated but not independent error terms. We also recall the results on the asymptotic behavior of the least squares estimator of weak FARIMA models obtained by  \cite{BMES2019}. This estimator will be used as the initial estimator in Le Cam's one-step procedure.

\subsection{Statement of the problem and notations}

Let $(X_t)_{t\in\mathbb{Z}}$ be a long memory second-order stationary process satisfying a weak FARIMA$(p,d_0,q)$ representation of the form 
\begin{equation}\label{FARIMA}
a(L)(1-L)^{d_0}X_t=b(L)\epsilon_t,
\end{equation}
where $d_0\in(0,1/2)$ is the long memory parameter, $({\epsilon}_t)_{t\in\mathbb Z}$ is a sequence of uncorrelated random variables defined on some probability space $(\Omega,\mathcal{T},\mathbb{P})$ with zero mean and common variance $\sigma_{\epsilon}^2$, $L$ stands for the back-shift operator and
$a(L)=1-\sum_{i=1}^pa_iL^i$, respectively $b(L)=1-\sum_{i=1}^qb_iL^i$, is the autoregressive, respectively the moving-average,  operator. These operators represent the short memory part of the model and are supposed to have all their roots outside the unit disk with no zero in common to ensure the invertibility of the model and the unique identifiability of the parameters.

\noindent The fractional difference operator $(1-L)^{d_0}$ is defined, using the generalized binomial series, by
\begin{equation*}
(1-L)^{d_0}=\sum_{j\geq 0}\alpha_j(d_0)L^j,
\end{equation*}
where for all $j\ge 0$, $\alpha_j(d_0)=\Gamma(j-d_0)/\left\lbrace \Gamma(j+1)\Gamma(-d_0)\right\rbrace $ and $\Gamma(\cdot)$ is the Gamma function. It can be readily shown, see for example \cite{Beran2013}, that for large $j$,  $\alpha_j(d_0)\sim j^{-d_0-1}/\Gamma(-d_0)$ . It is therefore clear that the fractional difference operator impacts the speed of convergence to 0 of the coefficients in the AR($\infty$) and MA($\infty$) representations of Model \eqref{FARIMA} compared with standard short-memory ARMA models where this operator is absent. This loss of speed, compared to the exponential one of ARMA models, implies that the series of autocovariances of the process $(X_t)_{t\in\mathbb{Z}}$ defined in \eqref{FARIMA} is not absolutely summable. 

\noindent Let $\Theta^{*}$ be the parameter space
\begin{align*}
\Theta^{*}=&  \Big
\{\theta=(\theta_1,\theta_2,\dots,\theta_{p+q})
\in{\Bbb R}^{p+q};  a_{\theta}(z)=1-\sum_{i=1}^{p}\theta_{i}z^{i}\, \text{
and } b_{\theta}(z)=1-\sum_{j=1}^{q}\theta_{p+j}z^{j} \\
& \hspace{2.5cm} \text{ have all
their zeros outside the unit disk}\Big \}.
\end{align*}
Denote by $\Theta$ the Cartesian product $\Theta^{*}\times (0,1/2)$. Note that the unknown parameter of interest $\theta_0=(a_1,a_2,\dots,a_p,b_1,b_2,\dots,b_q,d_0)^{'}$ belongs to the parameter space $\Theta$.

\noindent For all $\theta=(\theta_1,\theta_2,\dots,\theta_{p+q},d)^\prime\in\Theta$, we define $( \epsilon_t(\theta)) _{t\in\mathbb{Z}}$ as the second-order stationary process which is the solution of
\begin{equation}\label{FARIMA-th}
\epsilon_t(\theta)=\sum_{j\geq 0}\alpha_j(d)X_{t-j}-\sum_{i=1}^p\theta_i\sum_{j\geq 0}\alpha_j(d)X_{t-i-j}+\sum_{j=1}^q\theta_{p+j}\epsilon_{t-j}(\theta).
\end{equation}
Observe that, for all $t\in\mathbb{Z}$, $\epsilon_t(\theta_0)=\epsilon_t$ a.s.  Given a realization  $X_1,\dots,X_n$ of length $n$, $\epsilon_t(\theta)$ can be approximated, for $0<t\leq n$, by $\tilde{\epsilon}_t(\theta)$ defined recursively by
\begin{equation}\label{exp-epsil-tilde}
\tilde{\epsilon}_t(\theta)=\sum_{j=0}^{t-1}\alpha_j(d)X_{t-j}-\sum_{i=1}^p\theta_i\sum_{j=0}^{t-i-1}\alpha_j(d)X_{t-i-j}+\sum_{j=1}^q\theta_{p+j}\tilde{\epsilon}_{t-j}(\theta),
\end{equation}
with $\tilde{\epsilon}_t(\theta)=X_t=0$ if $t\leq 0$.

As shown in Proposition~\ref{prop:ecart-eps} (see Section~\ref{sec:proofs}), these initial values are asymptotically
negligible and in particular it holds that
$\epsilon_t(\theta)-\tilde{\epsilon}_t(\theta)\to 0$ almost-surely as $t\to\infty$ uniformly in $\theta$.
Thus the choice of the initial values has no influence on the asymptotic properties of the model parameters estimator.\\

\noindent Let $\Theta^{*}_\kappa$ denote the compact set
$$\Theta^{*}_\kappa=\left\lbrace \theta\in\mathbb{R}^{p+q}; \text{ the roots of the polynomials } a_{\theta}(z) \text{ and } b_{\theta}(z) \text{ have modulus } \geq 1+\kappa\right\rbrace.$$
We define the set $\Theta_\kappa$ as the Cartesian product $\Theta^{*}_\kappa\times\left[ d_1,d_2\right]$, where $\kappa$ is a positive constant chosen such that $\theta_0$ belongs to $\Theta_\kappa$ and where $\left[ d_1,d_2\right]\subset ( 0,1/2)$.\\

\noindent For $n\geq 1$ and $\theta\in\Theta$, consider the function
\begin{equation}\label{Qn}
Q_n(\theta)=\frac{1}{n}\sum_{t=1}^n\tilde{\epsilon}_t^2(\theta),
\end{equation}
where $(\tilde{\epsilon}_t(\theta))_{t\in\mathbb{Z}}$ is given in \eqref{exp-epsil-tilde}. The Le Cam one-step estimator is defined, almost-surely, by
\begin{equation}\label{eq: one-step est}
\overline{\theta}_n=\theta^{*}_n-\left\lbrace\frac{\partial^2}{\partial\theta\partial\theta^\prime}Q_n\left(\theta^{*}_n\right)\right\rbrace^{-1}\frac{\partial}{\partial\theta}Q_n\left(\theta^{*}_n\right),
\end{equation}
where $\theta^{*}_n$ is the least squares estimator of parameter $\theta_0$ calculated over the first $m=[ n^\delta]$, with $1/2<\delta\leq 1$, observations $X_1,\ldots,X_m$, \textit{i.e.}

\begin{equation}\label{eq: LSE}
\theta^{*}_n=\underset{\theta\in\Theta_\kappa}{\mathrm{argmin}} \ Q_m(\theta),
\text{ where }Q_m(\theta)=\frac{1}{[ n^\delta]}\sum_{t=1}^{[ n^\delta]}\tilde{\epsilon}_t^2(\theta).
\end{equation}
We will also propose alternative estimators where the matrix $\partial^2Q_n(\theta^{*}_n)/\partial\theta\partial\theta^\prime$ takes another forms (see Remark~\ref{rmq:hatJ} and Subsection~\ref{subsec:J}).
\subsection{Asymptotic properties}\label{sec:Asym-prop}

The asymptotic properties of the least squares estimator of the parameters of weak FARIMA models have been established by \cite{BMES2019}. The authors have showed, under some regularity assumptions on the noise, the consistency and the asymptotic normality of the least squares estimator. In this subsection, we study the asymptotic behavior of the Le Cam one-step estimator $\overline{\theta}_n$. We show, under the same assumptions, that the estimator $\overline{\theta}_n$ converges not only in probability but almost-surely to the true parameter $\theta_0$ and also satisfies a central limit theorem with a similar limit variance.

To ensure the strong consistency of the Le Cam one-step estimator $\overline{\theta}_n$, we assume that the innovation process in \eqref{FARIMA} satisfies the following condition:\\

\begin{itemize}
\item[{\bf (A1):}] The process $(\epsilon_t)_{t\in\mathbb{Z}}$ is strictly stationary and ergodic.\\
\end{itemize}
Our first main result is stated in the following theorem.

\begin{thm}{(Strong consistency).}\label{th:convergence} Assume that $(X_t)_{t\in\mathbb{Z}}$ satisfies \eqref{FARIMA}. Let $(\overline{\theta}_n)_{n\geq 1}$ be the sequence of Le Cam's one-step estimators defined by \eqref{eq: one-step est}. Under Assumption \textbf{(A1)}, we have
\begin{equation*}
\overline{\theta}_n\xrightarrow[n\to \infty]{a.s.} \, \theta_0.
\end{equation*}
\end{thm}
The proof of this theorem is given in Section~\ref{sec:proofs}.\\

\noindent For the asymptotic normality of the Le Cam one-step estimator, additional assumptions are required. It is necessary to assume
that $\theta_0$ is not on the boundary of the parameter space
${\Theta_\kappa}$.\\

\begin{itemize}
\item[{\bf (A2):}] We have $\theta_0\in\overset{\circ}{\Theta_\kappa}$, where $\overset{\circ}{\Theta_\kappa}$ denotes the interior of $\Theta_\kappa$.\\
\end{itemize}

\noindent The stationary process $(\epsilon_t)_{t\in\mathbb{Z}}$ is not supposed to be an independent sequence. So one needs to control its dependency by means of its  strong mixing coefficients $\left\lbrace \alpha_{\epsilon}(h)\right\rbrace _{h\geq 0}$ defined by
$$\alpha_{\epsilon}\left(h\right)=\sup_{A\in\mathcal F_{-\infty}^t,B\in\mathcal F_{t+h}^{\infty}}\left|\mathbb{P}\left(A\cap
B\right)-\mathbb{P}(A)\mathbb{P}(B)\right|,$$
where $\mathcal F_{-\infty}^t=\sigma (\epsilon_u, u\leq t )$ and $\mathcal F_{t+h}^{\infty}=\sigma (\epsilon_u, u\geq t+h)$.

\noindent We shall need  an integrability assumption on the moments of the noise $(\epsilon_t)_{t\in\mathbb{Z}}$ and a summability condition on the strong mixing coefficients $\{\alpha_{\epsilon}(h)\}_{h\geq 0}$.\\

\begin{itemize}
\item[{\bf (A3):}] There exists an integer $\tau$ such that for some $\nu\in(0,1]$, we have $\mathbb{E}|\epsilon_t|^{\tau+\nu}<\infty$ and $\sum_{h=0}^{\infty}(h+1)^{k-2} \left\lbrace \alpha_{\epsilon}(h)\right\rbrace ^{\frac{\nu}{k+\nu}}<\infty$ for $k=1,\dots,\tau$.\\
\end{itemize}
In order to state our asymptotic normality result, we define the function
\begin{align*}
O_n(\theta)& =\frac{1}{n}\sum_{t=1}^n\epsilon_t^2(\theta),
\end{align*}
where the sequence $\left( \epsilon_t(\theta)\right)_{t\in\mathbb{Z}}$ is given by \eqref{FARIMA-th}, and we consider the following information matrices
\begin{equation*}
I(\theta)=\lim_{n\rightarrow\infty}Var\left\lbrace \sqrt{n}\frac{\partial}{\partial\theta}O_n(\theta)\right\rbrace \text{  and  } J(\theta)=\lim_{n\rightarrow\infty}\left[ \frac{\partial^2}{\partial\theta_i\partial\theta_j}O_n(\theta)\right] \text{a.s.}
\end{equation*}
The existence of these matrices and the invertibility of $J(\theta_0)$ are proved in Lemmas 16 and 18 in \cite{BMES2019} for weak FARIMA.\\

\noindent Our second main result is given in the next theorem.

\begin{thm}{(Asymptotic normality).}\label{th:asy-norm} Assume that $(X_t)_{t\in\mathbb{Z}}$ satisfies \eqref{FARIMA}. Under \textbf{(A1)}, \textbf{(A2)} and \textbf{(A3)} with $\tau=4$, the sequence $\{\sqrt{n}(\overline{\theta}_n-\theta_0)\}_{n\geq 1}$ has a limiting centered normal distribution with covariance matrix $\Omega:=J^{-1}(\theta_0)I(\theta_0)J^{-1}(\theta_0)$.
\end{thm}
The detailed proof of this result is postponed to Section~\ref{sec:proofs}.\\

\begin{rmq}\label{rmq:hatJ} The quantity $\partial^2Q_n(\theta^{*}_n)/\partial\theta\partial\theta^\prime$ in the definition of Le Cam's one-step estimator \eqref{eq: one-step est} can be replaced by $J(\theta^{*}_n)$ since the matrix $J(\theta_0)$ has an explicit expression in the framework of FARIMA models (see Subsection~\ref{subsec:J}).\\
It can also be replaced by
\begin{equation*}
\hat{J}_n(\theta^{*}_n)=\frac{2}{n}\sum_{t=1}^n\left\lbrace \frac{\partial}{\partial\theta}\tilde{\epsilon}_t\left(\theta^{*}_n\right)\right\rbrace \left\lbrace \frac{\partial}{\partial\theta^{'}}\tilde{\epsilon}_t\left(\theta^{*}_n\right)\right\rbrace .
\end{equation*}
This is due to the fact that $\hat{J}_n(\cdot)$ satisfies a stochastic Lipschitz condition similar to the one in Proposition~\ref{prop:lipschitz} and that $\hat{J}_n(\theta_n^{*})$ converges almost-surely to $J(\theta_0)$ (see Lemma \ref{lemme:der2Qn}). The ergodic theorem and the uncorrelatedness of $(\epsilon_t)_{t\in\mathbb{Z}}$ are behind the intuition of the construction of the estimator $\hat{J}_n(\theta_n^{*})$. More precisely, observe that under \textbf{(A1)}, the matrix $J(\theta_0)$ can be rewritten as 
\begin{align}\label{eq:J}
J(\theta_0)&=\lim_{n\rightarrow\infty}\left\lbrace\frac{2}{n}\sum_{t=1}^n\frac{\partial}{\partial\theta}\epsilon_t(\theta_0)\frac{\partial}{\partial\theta^{'}}\epsilon_t(\theta_0)+\frac{2}{n}\sum_{t=1}^n\epsilon_t(\theta_0)\frac{\partial^2}{\partial\theta\partial\theta^{'}}\epsilon_t(\theta_0)\right\rbrace\nonumber\\
&=2\mathbb{E}\left[ \frac{\partial}{\partial\theta}\epsilon_t(\theta_0)\frac{\partial}{\partial\theta'}\epsilon_t(\theta_0)\right]+2\mathbb{E}\left[\epsilon_t(\theta_0)\frac{\partial^2}{\partial\theta\partial\theta'}\epsilon_t(\theta_0)\right]\nonumber\\
&=2\mathbb{E}\left[ \frac{\partial}{\partial\theta}\epsilon_t(\theta_0)\frac{\partial}{\partial\theta'}\epsilon_t(\theta_0)\right] \mathrm{ a.s.}
\end{align}
\end{rmq}

\begin{rmq} Under \textbf{(A1)}, \textbf{(A2)} and {\bf (A3)} with  $\tau=4$, it can be shown (see \cite{BMES2019}, Lemma 18) that the sequence $(\mathbb{E}\left[H_1(\theta_0)H_{1+k}^{'}(\theta_0)\right])_{k\in\mathbb{Z}}$ where, for all $t\in\mathbb{Z}$, $H_t(\theta)=2\epsilon_t(\theta)\frac{\partial}{\partial\theta}\epsilon_t(\theta) =(2\epsilon_t(\theta)\frac{\partial}{\partial\theta_1}\epsilon_t(\theta),\dots,2\epsilon_t(\theta)\frac{\partial}{\partial\theta_{p+q+1}}\epsilon_t(\theta)) ^\prime$, is absolutely summable. Therefore, from the stationarity of the centered process $(H_t(\theta_0))_{t\in\mathbb{Z}}$, we have

\begin{align*}
I(\theta_0)&=\lim_{n\rightarrow\infty}\mathrm{Var}\left( \frac{1}{\sqrt{n}}\sum_{t=1}^nH_t(\theta_0)\right)\nonumber\\
&=\lim_{n\rightarrow\infty}\frac{1}{n}\sum_{t=1}^n\sum_{s=1}^n\mathrm{Cov}\left( H_t(\theta_0),H_s(\theta_0)\right)\nonumber\\
&=\lim_{n\rightarrow\infty}\frac{1}{n}\sum_{k=1-n}^{n-1}(n-\left|k\right|)\mathrm{Cov}\left( H_1(\theta_0),H_{1+k}(\theta_0)\right)\nonumber\\
&=\sum_{k=-\infty}^{\infty}\mathbb{E}\left[H_1(\theta_0)H_{1+k}^{'}(\theta_0)\right].
\end{align*}
When the noise $(\epsilon_t)_{t\in\mathbb{Z}}$ is assumed to be an iid sequence, one can use the orthogonality of $\epsilon_t$ with any linear combination of $(\epsilon_s)_{s\leq t-1}$ in particular $\partial\epsilon_t(\theta_0)/\partial\theta$ (see Subsection~\ref{prelim}) to obtain 
\begin{align*}
I(\theta_0)&=\mathbb{E}\left[H_1(\theta_0)H_{1}^{'}(\theta_0)\right]+2\sum_{k=1}^{\infty}\mathbb{E}\left[H_1(\theta_0)H_{1+k}^{'}(\theta_0)\right]\\
&=4\mathbb{E}\left[\epsilon_1^2\frac{\partial}{\partial\theta}\epsilon_1(\theta_0)\frac{\partial}{\partial\theta^\prime}\epsilon_{1}(\theta_0)\right]+4\sum_{k=1}^{\infty}\mathbb{E}\left[\epsilon_1\frac{\partial}{\partial\theta}\epsilon_1(\theta_0)\epsilon_{1+k}\frac{\partial}{\partial\theta^\prime}\epsilon_{1+k}(\theta_0)\right]\\
&=2\sigma_{\epsilon}^2J(\theta_0).
\end{align*}
Thus, the asymptotic covariance matrix in the strong FARIMA case is reduced to $\Omega_S:=2\sigma_{\epsilon}^2J^{-1}(\theta_0)$. Generally, when
the noise is not an independent sequence, this simplification can not be made and we have $I(\theta_0)\ne 2\sigma_{\epsilon}^2J(\theta_0)$.
The true asymptotic covariance matrix $\Omega=J^{-1}(\theta_0)I(\theta_0)J^{-1}(\theta_0)$ obtained
in the weak FARIMA framework can be very different from $\Omega_S$.
\end{rmq}

\noindent A key point allowing to establish the limit distribution of the one-step estimator of the parameters of the weak FARIMA model \eqref{FARIMA} is the fact that $\theta\in\Theta_\kappa\longrightarrow\partial^2Q_n(\theta)/\partial\theta\partial\theta^\prime\in\mathbb{R}^{(p+q+1)\times (p+q+1)}$ is a stochastic Lipschitz function.

\begin{prop}\label{prop:lipschitz} Assume that $(X_t)_{t\in\mathbb{Z}}$ satisfies \eqref{FARIMA}. For any $i,j\in\{1,\ldots,p+q+1\}$ and all $\theta^{(1)}, \theta^{(2)}\in\Theta_\kappa$, one has 
$$\left|\frac{\partial^2}{\partial\theta_i\partial\theta_j}Q_n\left(\theta^{(1)}\right)-\frac{\partial^2}{\partial\theta_i\partial\theta_j}Q_n\left(\theta^{(2)}\right)\right|\leq \Delta_n\left\|\theta^{(1)}-\theta^{(2)}\right\|,$$
where $\Delta_n$ is bounded in probability.
\end{prop}
The proof of this proposition is detailed in Section~\ref{sec:proofs}.\\

\subsection{Explicit computations of $J(\theta)$}\label{subsec:J}
The particular structure of FARIMA models allows an explicit calculation of the matrix $J(\theta_0)$. Thus, the use of the closed form of the limit matrix $J(\theta_0)$ in \eqref{eq: one-step est} instead of the second derivative of the function $Q_n(\cdot)$  further improves the computational performance of Le Cam's one-step estimator while maintaining the same asymptotic properties. The matrix $J(\theta)$ clearly depends on the derivative of the process $(\epsilon_t(\theta))_{t\in\mathbb{Z}}$. This derivative can be expressed as an infinite linear combination of the past and present values of the true noise $(\epsilon_t)_{t\in\mathbb{Z}}$, which subsequently gives rise to a simple calculation of $J(\theta)$ by exclusively exploiting the uncorrelatedness of the innovations process $(\epsilon_t)_{t\in\mathbb{Z}}$. Let us be more precise. Observe that from Equations \eqref{eq:J} and \eqref{deriveesecepsil}, one has
\begin{align*}
J(\theta)&=2\sigma_\epsilon^2\sum_{i\geq 1}\begin{pmatrix}
\overset{\textbf{.}}{\lambda}_{i,1}^2\left( \theta\right)&\overset{\textbf{.}}{\lambda}_{i,1}\left( \theta\right)\overset{\textbf{.}}{\lambda}_{i,2}\left( \theta\right)&\cdots&\overset{\textbf{.}}{\lambda}_{i,1}\left( \theta\right)\overset{\textbf{.}}{\lambda}_{i,p+q+1}\left( \theta\right)\\
\overset{\textbf{.}}{\lambda}_{i,2}\left( \theta\right)\overset{\textbf{.}}{\lambda}_{i,1}\left( \theta\right)&\overset{\textbf{.}}{\lambda}_{i,2}^2\left( \theta\right)&\cdots&\overset{\textbf{.}}{\lambda}_{i,2}\left( \theta\right)\overset{\textbf{.}}{\lambda}_{i,p+q+1}\left( \theta\right)\\
\vdots&\vdots&\vdots&\vdots\\
\overset{\textbf{.}}{\lambda}_{i,p+q+1}\left( \theta\right)\overset{\textbf{.}}{\lambda}_{i,1}\left( \theta\right)&\overset{\textbf{.}}{\lambda}_{i,p+q+1}\left( \theta\right)\overset{\textbf{.}}{\lambda}_{i,2}\left( \theta\right)&\cdots&\overset{\textbf{.}}{\lambda}_{i,p+q+1}^2\left( \theta\right)
\end{pmatrix}.
\end{align*}
Furthermore, for  $1\le k \le p+q+1 $, the sequence $\overset{\textbf{.}}{\lambda}_{k}(\theta)= (\overset{\textbf{.}}{\lambda}_{i,k}(\theta))_{i\geq1}$ in Equation \eqref{deriveesecepsil} is the sequence of  the coefficients in the power series of
$$\frac{\partial}{\partial \theta_k }\left (  b_\theta^{-1}(z) a_{\theta}(z)(1-z)^{d-d_0} a_{\theta_0}^{-1}(z)b_{\theta_0}(z)\right ).$$ Thus, $\overset{\textbf{.}}{\lambda}_{i,k}\left(\theta_0\right)$ is the $i-$th coefficient taken in $\theta=\theta_0$. There are three cases.
\begin{itemize}
\item[$\diamond$] $k=1,\dots,p$:  \\
Since
$$\frac{\partial}{\partial \theta_k }\left (  b_\theta^{-1}(z) a_{\theta}(z)(1-z)^{d-d_0} a_{\theta_0}^{-1}(z)b_{\theta_0}(z) \right )= -b_\theta^{-1}(z) z^k (1-z)^{d-d_0} a_{\theta_0}^{-1}(z)b_{\theta_0}(z), $$ we deduce that $\overset{\textbf{.}}{\lambda}_{i,k}\left(\theta_0\right)$ is the $i-$th coefficient of $-z^k a_{\theta_0}^{-1}(z)$.\\

\item[$\diamond$] $k=p+1,\dots,p+q$:  \\
We have
$$\frac{\partial}{\partial \theta_k } \left (b_\theta^{-1}(z) a_{\theta}(z)(1-z)^{d-d_0} a_{\theta_0}^{-1}(z)b_{\theta_0}(z)\right ) = \left  (\frac{\partial}{\partial \theta_k } b_\theta^{-1}(z)\right  ) a_{\theta}(z) (1-z)^{d-d_0} a_{\theta_0}^{-1}(z)b_{\theta_0}(z) $$ and consequently $\overset{\textbf{.}}{\lambda}_{i,k}\left(\theta_0\right)$ is the $i-$th coefficient of $(\frac{\partial}{\partial \theta_k } b_{\theta_0}^{-1}(z) ) b_{\theta_0}(z)$.\\

\item[$\diamond$] $k=p+q+1$:  \\
In this case, $\theta_k=d$ and so we have
$$\frac{\partial}{\partial \theta_k } \left (b_\theta^{-1}(z) a_{\theta}(z)(1-z)^{d-d_0} a_{\theta_0}^{-1}(z)b_{\theta_0}(z)\right ) = b_\theta^{-1}(z)a_{\theta}(z) \mathrm{ln}(1-z)(1-z)^{d-d_0} a_{\theta_0}^{-1}(z)b_{\theta_0}(z) $$ which implies that $\overset{\textbf{.}}{\lambda}_{i,k}\left(\theta_0\right)$ is the $i-$th coefficient of $\mathrm{ln}(1-z)$ which is equal to $-1/i$.
\end{itemize}
\begin{exemple} In this example, we illustrate the previous calculations in the case of weak FARIMA$(1,d,1)$ model (\textit{i.e.} when $p=q=1$ in \eqref{FARIMA} and \eqref{FARIMA-th}). This model is widely used in practice. Since the modulus of the autoregressive parameter $a_0$ and the moving-average parameter $b_0$ are assumed to be strictly less than 1, one can easily obtain that $a^{-1}(z)=(1-a_0z)^{-1}=\sum_{i\geq 0}a_0^iz^i$ and similarly $b^{-1}(z)=\sum_{i\geq 0}b_0^iz^i$. So, it can be shown that 
\begin{equation*}
\frac{\partial}{\partial\theta}\epsilon_t(\theta_0)=\begin{pmatrix}
\vspace*{0.2cm}
\frac{\partial}{\partial a}\epsilon_t(\theta_0)\\
\vspace*{0.2cm}
\frac{\partial}{\partial b}\epsilon_t(\theta_0)\\
\frac{\partial}{\partial d}\epsilon_t(\theta_0)
\end{pmatrix}=-\sum_{i\geq 1}\begin{pmatrix}
a_0^{i-1}\\
-b_0^{i-1}\\
\frac{1}{i}
\end{pmatrix}\epsilon_{t-i}
\end{equation*}
and consequently, we deduce that 
\begin{equation*}
J(\theta_0)=2\sigma_\epsilon^2\begin{pmatrix}
1/(1-a_0^2)&-1/(1-a_0b_0)&-\ln(1-a_0)/a_0\\
-1/(1-a_0b_0)&1/(1-b_0^2)&\ln(1-b_0)/b_0\\
-\ln(1-a_0)/a_0&\ln(1-b_0)/b_0&\pi^2/6
\end{pmatrix}.
\end{equation*}
This explicit expression can be used in the one-step procedure to speed it up.
\end{exemple}

\section{Numerical illustrations}\label{sec:sim}
We investigate in this section the behavior of the one-step estimator on finite sample sizes through Monte Carlo experiments. The numerical illustrations of this section are made with the open source statistical software \textsf{R}  (\cite{R20}).
\subsection{Simulation studies}
The behavior of Le Cam's one-step estimator is numerically studied for FARIMA$(1,d,1)$ model of the form
\begin{equation}\label{FARIMA-sim}
(1-L)^d(X_t-aX_{t-1})=\epsilon_t-b\epsilon_{t-1},
\end{equation}
where the unknown parameter $(a,b,d)^\prime$ takes different values. We start by comparing the asymptotic properties of the one-step estimator and the LSE in both strong and weak frameworks. For this, firstly we consider that the innovation process $(\epsilon_t)_{t\in\mathbb{Z}}$ in \eqref{FARIMA-sim} is an iid centered Gaussian process with common variance 1 (which corresponds to the strong FARIMA case) and secondly that it is defined by
\begin{equation}\label{weak-noise}
\epsilon_{t}=\eta_t^2\eta_{t-1},
\end{equation}
where  $(\eta_t)_{t\in\mathbb{Z}}$ is a sequence of iid centered Gaussian random variables with variance 1. Note that the innovation process in \eqref{weak-noise} is a weak white noise which is not a martingale difference.\\

\noindent The Figure~\ref{fig1} compares the empirical distribution of the LSE and the Le Cam one-step estimator of the memory parameter $d$ in the strong case (first column) and the weak case (second column). We simulated $M=2,000$ independent trajectories of size $n=5,000$ of Model \eqref{FARIMA-sim} with $(a,b,d)^\prime=(0.2,0.5,0.3)^\prime$ endowed first by the strong Gaussian noise and then by the weak noise \eqref{weak-noise}. We considered that $\delta=0.9$. Let us recall that the fraction $\delta$ defines the size of the sample on which the initial estimator is calculated.\\

\noindent The LSE calculated on the fraction $\delta$ of the data (see the two middle graphs) is computed faster than the LSE on the whole sample but is naturally less efficient. We can observe the similarity of the empirical distributions of the one-step estimator and the LSE on the whole sample. This perfectly illustrates the theoretical results presented in Subsection~\ref{sec:Asym-prop}.\\

\noindent In Figure \ref{fig2}, we present the empirical distribution of the Le Cam one-step estimator of the memory parameter for different values of the parameters in \eqref{FARIMA-sim} induced by the noise \eqref{weak-noise} with $n=5,000$ over the 2,000 replications. This graph highlights the adequacy of the empirical results (distributions over a finite sample size) and the theoretical results obtained in Theorem~\ref{th:asy-norm}, even when the parameter is close to the boundary of the parameter space.\\

\noindent Finally, we compare in Figure \ref{fig3} the computation time (in seconds), with respect to the sample size, of the LSE and the one-step estimator of all the parameters, with two different fractions ($\delta=0.7$ and $\delta=0.9$) and $(a,b,d)^\prime=(0.2,0.5,0.3)^\prime$. For each size $n$, we simulated 10 replications to calculate the estimators. We observe that the one-step estimator outperforms the LSE in terms of computation time. It should also be noted that taking a small fraction $\delta$ further reduces the calculation time.

\begin{figure}[H]
\begin{center}\includegraphics[width=14cm]{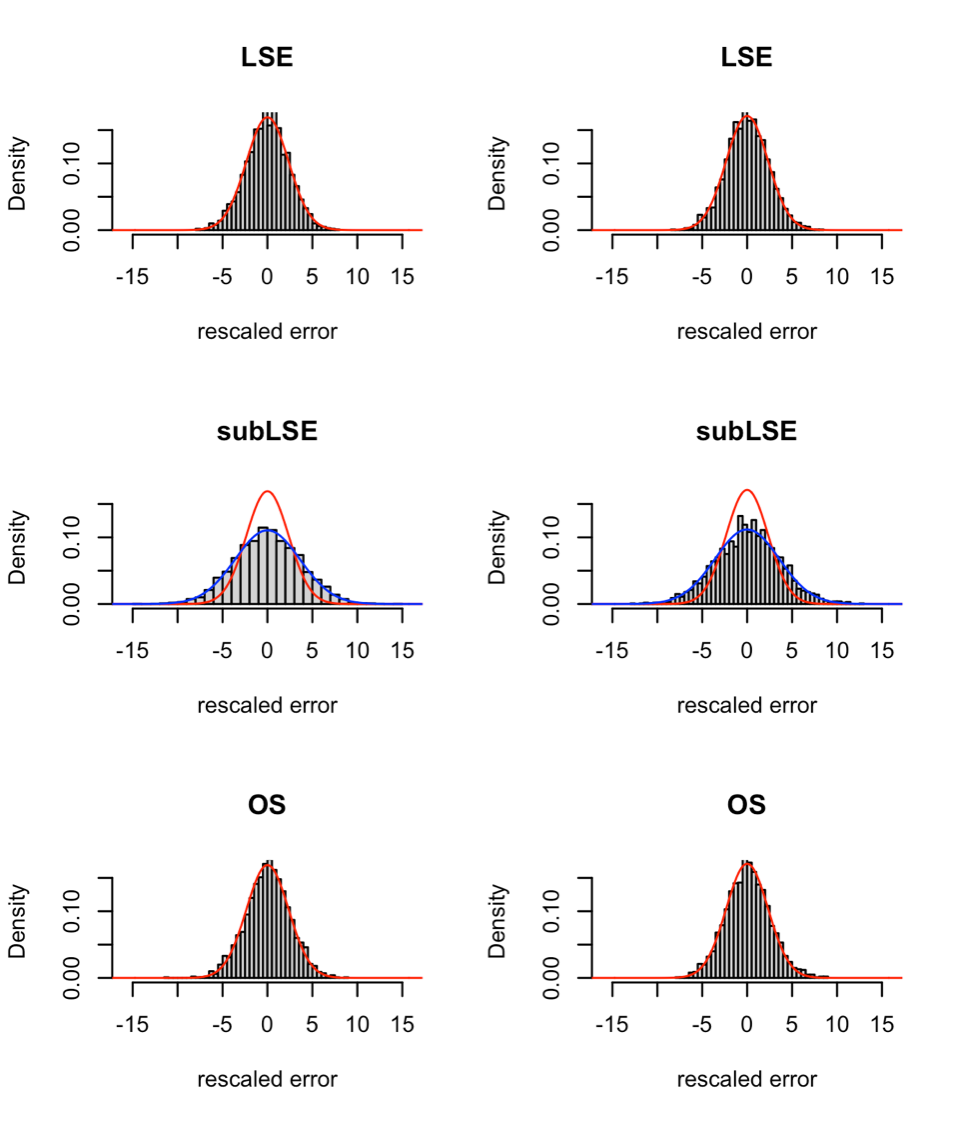}
\caption{Histograms for the $M=2,000$ Monte Carlo simulations of the rescaled statistical error of the LSE (first line), the LSE calculated on the fraction $\delta=0.9$ of the data (second line) and one-step estimator (last line) of the memory parameter for Model \eqref{FARIMA-sim} with $(a,b,d)^\prime=(0.2,0.5,0.3)^\prime$ and $n=5,000$. Superimposed red and blue lines are the theoretical centered Gaussian asymptotic distributions of the LSE and the subLSE, respectively.}
\label{fig1}
\end{center}
\end{figure}

\begin{figure}[H]
\begin{center}\includegraphics[width=14cm]{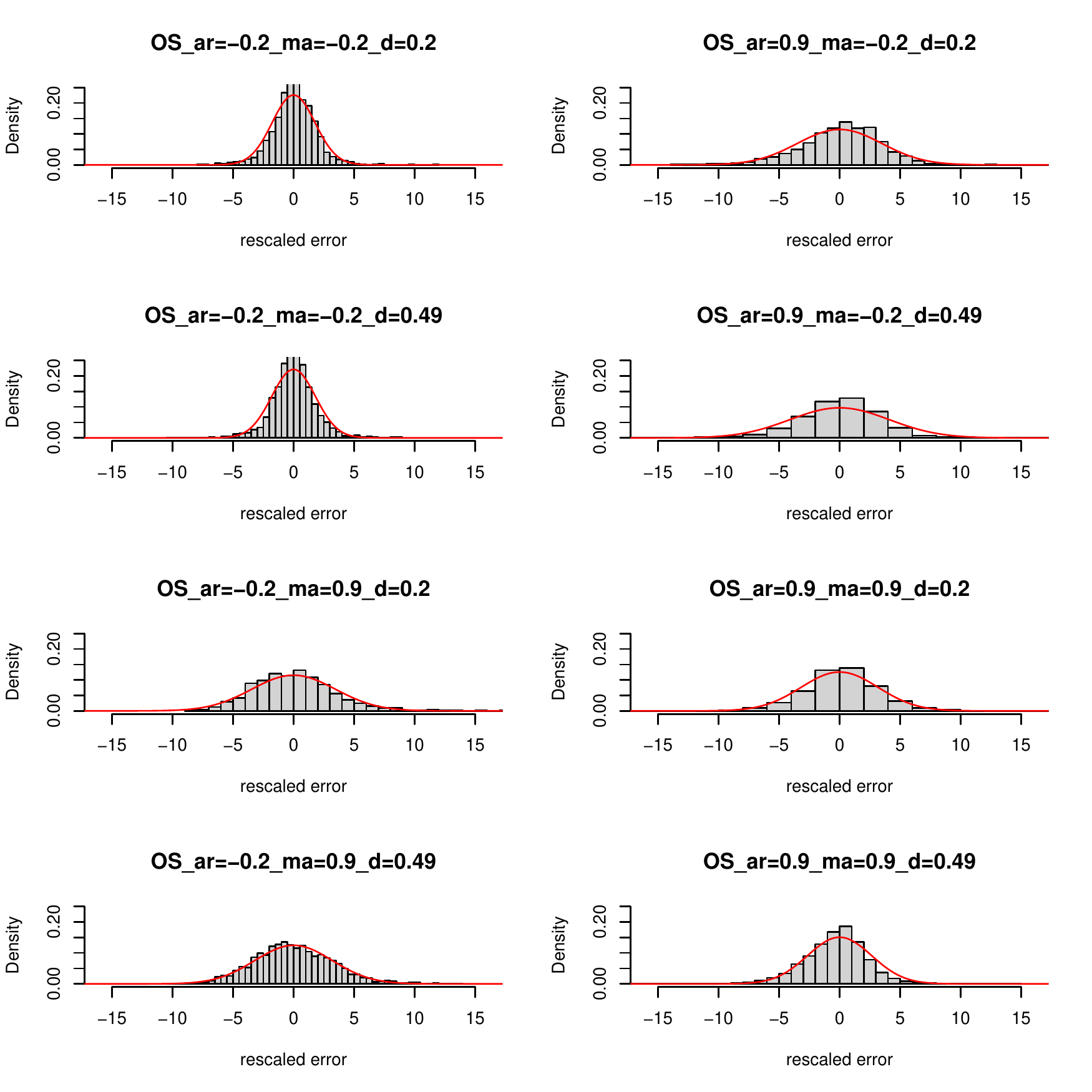}
\caption{Empirical distribution over $M=2,000$ simulations of the rescaled statistical error of the one-step estimator of the memory parameter for FARIMA$(1,d,1)$ model induced by the weak white noise \eqref{weak-noise} with $n=5,000$ and different values of the parameters.}
\label{fig2}
\end{center}
\end{figure}

\begin{figure}[H]
\begin{center}\includegraphics[width=14cm]{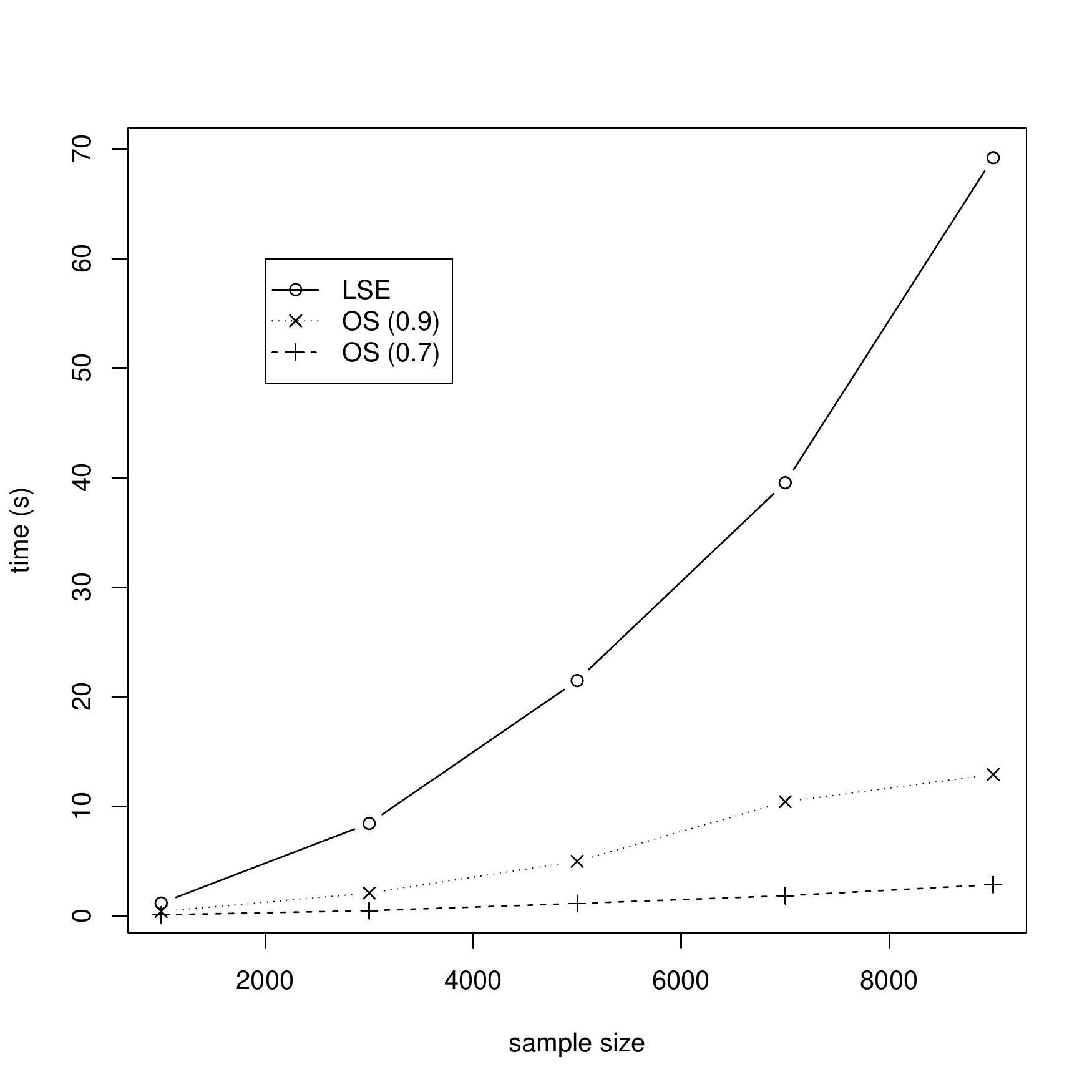}
\caption{Comparison of the computation times with respect to the sample size of the LSE and the one-step estimator of the parameters of Model \eqref{FARIMA-sim} induced by Noise \eqref{weak-noise} with $(a,b,d)^\prime=(0.2,0.5,0.3)^\prime$. For each size $n$, 10 replications are generated.}
\label{fig3}
\end{center}
\end{figure}

\subsection{Illustrative example}

We consider in this example the daily $\log$ returns of the Standard \& Poor's 500 index (S\&P~500, for short). The $\log$ returns (or simply returns) are defined by $r_t=100\log(p_t/p_{t-1})$ where $p_t$ is the price index of the S\&P~500 at time $t$. The observations of the S\&P~500 index cover the period from January 3, 1950 to February 14, 2019. The length of the series is $n=17,391$. The data can be downloaded with the \verb!R! package \verb!quantmod!.\\

The phenomenon of long memory has been widely studied for financial series. \cite{DGE1993} have shown that the positive powers of the absolute value of returns have more persistence than the returns themselves. We choose here the case of squared returns. The mean and the standard deviation of $(r_t^2)_{t\ge 1}$ are $0,9347$ and $5,0036$. As in \cite{Ling2003}, we consider the centered series $(X_t)_{t\ge 1}$ of the squared returns, that is, $X_t=r_t^2-0,9347$.The sample autocorrelations of the series $(X_t)_{t\ge 1}$ (see Figure \ref{fig5}) decrease very slowly and are approximated by the function in blue (which is not integrable on $\mathbb{R}$). This suggests that this series has a long memory. 

\begin{figure}[H]
\begin{minipage}{0.48\textwidth}
\vspace*{-0.4cm}
\includegraphics[width=7cm]{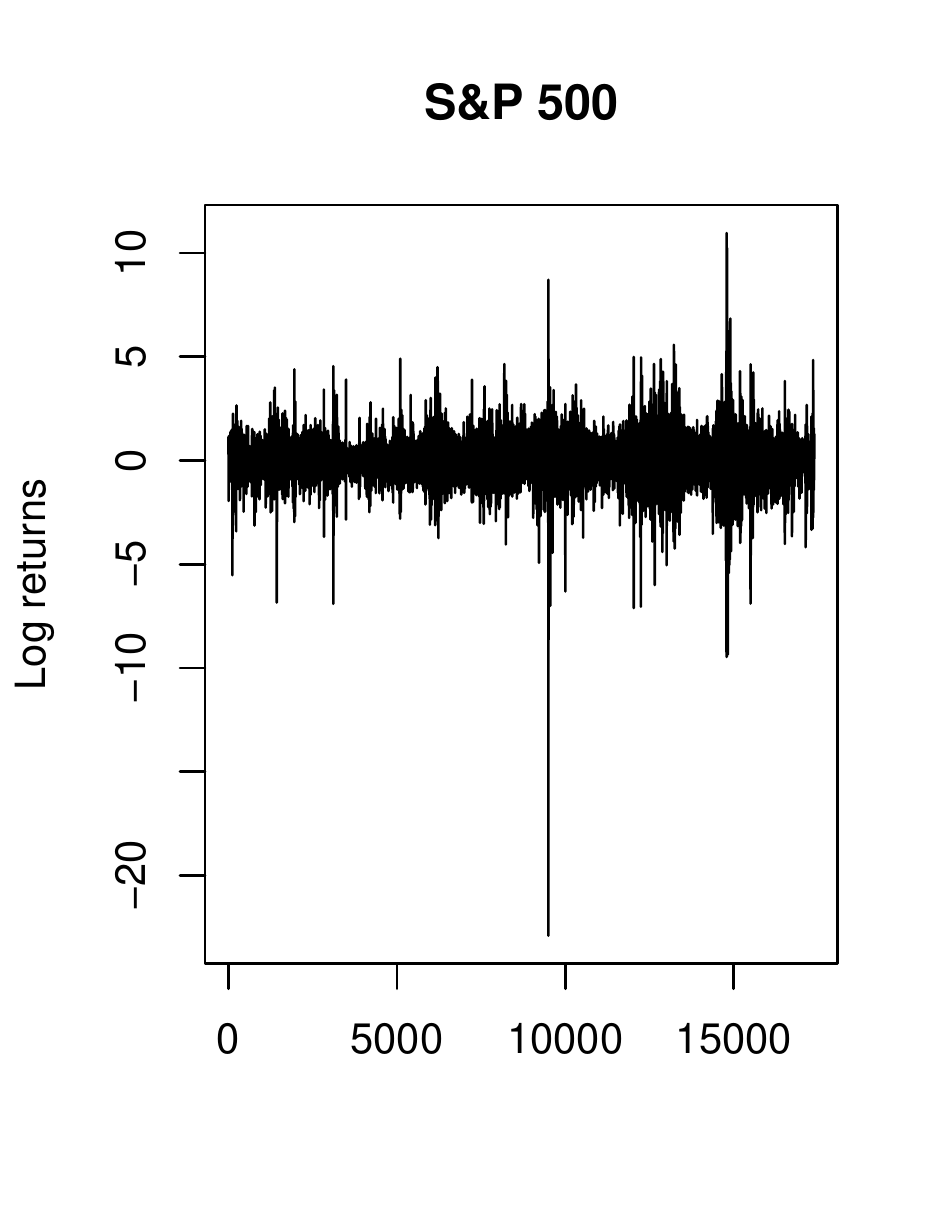}
\caption{Returns of the S\&P~500 index from January 3, 1950 to February 14, 2019.}
\label{fig4}
\end{minipage}
\hspace*{0.3cm}
\begin{minipage}{0.48\textwidth}
\includegraphics[width=7cm]{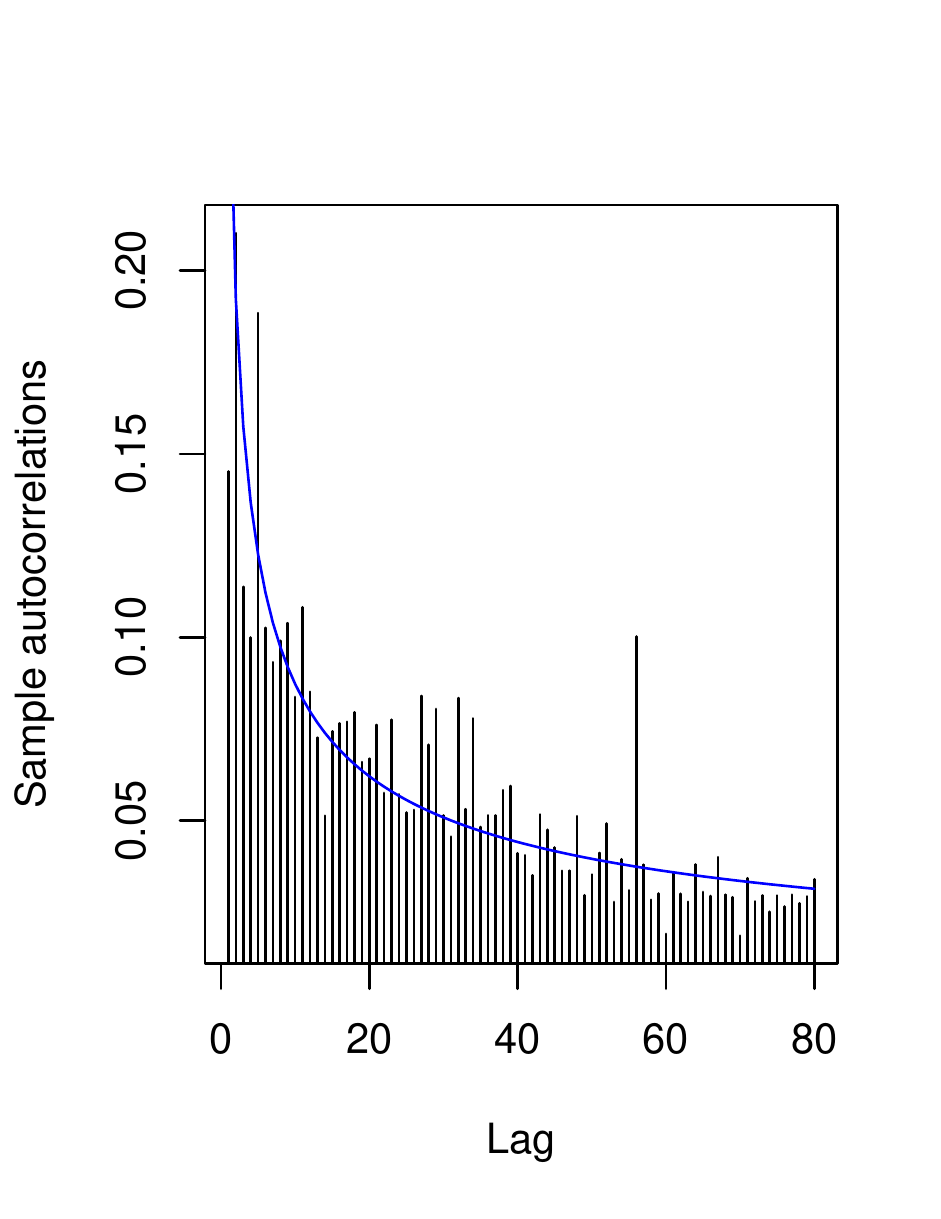}
\caption{Sample autocorrelations of squared returns of the S\&P~500 stock market index. The curve in blue is that of $x\rightarrow 0.26/x^{\textcolor{blue}{\mathbf{0.49}}}$.}
\label{fig5}
\end{minipage}
\end{figure}

It has been statistically validated in our previous work (see \cite{bes--2019}) that the time series $(X_t)_{t\ge 1}$ could be adjusted by a weak FARIMA$(1,d,1)$ model. It is worth emphasizing that the strong FARIMA fitting with the same orders is rejected for this series. The calibration of such models with the LSE is time consuming. Consequently, we propose a remote solution (API REST with the HTTP protocol) to provide fast and similar estimates with the same asymptotic properties as the LSE. In \verb!R!, we can access to this API with the \verb!httr! package using the \verb!POST! command and format our data into \verb!JSON! (package \verb!jsonlite!). We can find an example of a program at the address \url{https://www.effi-stats.fr/files/API_wFARIMA.R}.

\section{Proofs}\label{sec:proofs}

Before starting the proofs of our main results, we introduce in the next subsection some necessary results on estimations of the coefficients of formal power series that will arise in our study.

\noindent In all our proofs, $K$ is a positive constant that may vary from line to line.

\subsection{Preliminary results}\label{prelim}
We begin by recalling the following properties on power series. If for $|z|\le R$, the power series $f(z)=\sum_{i\ge 0}a_iz^i$ and  $g(z)=\sum_{i\ge 0}b_iz^i$ are well defined, then $(fg)(z)= \sum_{i\ge 0} c_iz^i$ is also well defined for $|z|\le R$ with the sequence $(c_i)_{i\ge 0}$ which is given by $c=a\ast b$ where $\ast$ denotes the convolution product between $a$ and $b$ defined by $c_i=\sum_{k=0}^i a_kb_{i-k}=\sum_{k=0}^i a_{i-k}b_{k}$.

Now we come back to the power series that arise in our context. Remind that for the true value of the parameter,
\begin{equation}\label{FF}
a_{\theta_0}(L)(1-L)^{d_0}X_t=b_{\theta_0}(L)\epsilon_t.
\end{equation}
Thanks to the assumptions on the moving average polynomials $b_\theta$ and the autoregressive polynomials $a_\theta$, the power series $a_\theta^{-1}$ and $b_\theta^{-1}$ are well defined.

Thus the functions $\epsilon_t(\theta)$ defined in \eqref{FARIMA-th}  can be written as
\begin{align}\label{epsi}
\epsilon_t(\theta) &= b^{-1}_{\theta}(L) a_{\theta}(L)(1-L)^{d}X_t \\
\label{epsi-bis}
& =b^{-1}_{\theta}(L) a_{\theta}(L)(1-L)^{d-d_0}a^{-1}_{\theta_0}(L) b_{\theta_0}(L)\epsilon_t
\end{align}
and if we denote $\gamma(\theta)=(\gamma_i(\theta))_{i\ge 0}$ the sequence of coefficients of the power series $b^{-1}_{\theta}(z) a_{\theta}(z)(1-z)^{d}$, we may write for all $t\in\mathbb Z$,
\begin{align}\label{AR-Inf}
\epsilon_t(\theta)&=\sum_{i\geq 0}\gamma_i(\theta)X_{t-i}.
\end{align}
In the same way, by \eqref{epsi} we have
\begin{align*}
X_t & = (1-L)^{-d}a^{-1}_{\theta}(L) b_{\theta}(L)\epsilon_t(\theta)
\end{align*}
and if we denote $\eta(\theta)=(\eta_i(\theta))_{i\ge 0}$ the coefficients of the power series $(1-z)^{-d}a^{-1}_{\theta}(z) b_{\theta}(z)$ one has
\begin{align}
\label{MA-Inf}
X_t & = \sum_{i\geq 0}\eta_i(\theta)\epsilon_{t-i}(\theta).
\end{align}
We strength the fact that $\gamma_0(\theta)=\eta_0(\theta)=1$ for all $\theta$. 

For large $j$, \cite{hallin} have shown that uniformly in $\theta$ the sequences  $\gamma(\theta)$ and $\eta(\theta)$
satisfy
\begin{equation}\label{Coef-Gamma}
\frac{\partial^k\gamma_j(\theta)}{\partial\theta_{i_1}\cdots\partial\theta_{i_k}}=\mathrm{O}\left( j^{-1-d}\left\lbrace \log(j)\right\rbrace ^k\right),\text{ for }k=0,1,2,3,
\end{equation}
and
\begin{equation}\label{Coef-Eta}
\frac{\partial^k\eta_j(\theta)}{\partial\theta_{i_1}\cdots\partial\theta_{i_k}}=\mathrm{O}\left( j^{-1+d}\left\lbrace \log(j)\right\rbrace ^k\right), \text{ for }k=0,1,2,3.
\end{equation}

One difficulty that has to be addressed is that \eqref{AR-Inf} includes the infinite past $(X_{t-i})_{i\ge 0}$ whereas only a finite number of observations $(X_t)_{1\leq t\leq n}$ are available to compute the estimators defined in \eqref{eq: one-step est}.
The simplest solution is truncation which amounts to setting all unobserved values equal to zero. Thus, for all $\theta\in\Theta$ and $1\le t\le n$ one defines
\begin{equation}\label{epsilon-tilde--}
\tilde{\epsilon}_t(\theta)=\sum_{i=0}^{t-1}\gamma_i(\theta)X_{t-i}= \sum_{i\ge 0} \gamma_i^t(\theta) X_{t-i},
\end{equation}
where the truncated sequence $\gamma^t(\theta)= ( \gamma_i^t(\theta))_{i\ge 0}$ is defined by
$$\gamma_i^t(\theta)=\left\{
\begin{array}{rl}
 \gamma_i(\theta) &\text{ if } \ 0\leq i\leq t-1\ , \\
0& \text{ otherwise.}
\end{array}\right.$$

In the following proposition, we show that the difference between $\epsilon_t(\theta)$ and $\tilde{\epsilon}_t(\theta)$ converges almost-surely to 0 as $t\to\infty$ and this uniformly in $\theta$. This proposition shows that the convergence of the least squares estimator $\theta^{*}_n$ in \eqref{eq: LSE} studied in \cite{BMES2019} is not only in probability but it is almost-sure when $d_0\in(0,1/2)$. This last confirmation can be easily demonstrated by following line by line the proof of Theorem 1 in \cite{fz98}.

\begin{prop}\label{prop:ecart-eps} Let $(X_t)_{t\in\mathbb{Z}}$ be the second-order stationary process given by \eqref{FARIMA}. Under the standard assumptions of invertibility and identifiability on the autoregressive polynomial $a$ and the moving-average polynomial $b$, we have almost-surely
\begin{equation}\label{conv:ecart-eps}
\lim_{t\rightarrow \infty} \sup_{\theta\in\Theta_{\delta}}\left| \epsilon_t(\theta)-\tilde{\epsilon}_t(\theta)\right|=0.
\end{equation}
\end{prop}
\begin{proof}
From \eqref{AR-Inf} and \eqref{epsilon-tilde--},  it can be readily shown that for all $\theta\in\Theta_{\delta}$ and any $t\in\mathbb{Z}$, 
\begin{align*}
\epsilon_t(\theta)-\tilde{\epsilon}_t(\theta)&=\sum_{j\geq 0}\gamma_j(\theta)X_{t-j}-\sum_{j=0}^{t-1}\gamma_j(\theta)X_{t-j}=\sum_{j\geq t}\gamma_j(\theta)X_{t-j}=\sum_{k\geq 0}\gamma_{t+k}(\theta)X_{-k}.
\end{align*}
Recall that for any sequence $(Y_n)_{n\geq 0}$ of random variables it holds that
$$Y_n\underset{n\to\infty}{\overset{\text{a.s.}}{\longrightarrow}}Y\Leftrightarrow\sup_{k\geq n}\left|Y_k-Y\right|\underset{n\to\infty}{\overset{\mathbb{P}}{\longrightarrow}}0.$$

\noindent Hence $\sup_{\theta\in\Theta_{\delta}}\left| \epsilon_t(\theta)-\tilde{\epsilon}_t(\theta)\right|$ converges almost-surely to $0$ as soon as $$\sup_{k\geq t}\sup_{\theta\in\Theta_{\delta}}\left| \epsilon_k(\theta)-\tilde{\epsilon}_k(\theta)\right|$$  converges in probability to $0$. In view of \eqref{Coef-Gamma}, one has for all $\beta>0$ and for large $t$,
\begin{align*}
\mathbb{P}\left( \sup_{k\geq t}\sup_{\theta\in\Theta_{\delta}}\left| \epsilon_k(\theta)-\tilde{\epsilon}_k(\theta)\right|>\beta\right) &=\mathbb{P}\left( \sup_{k\geq t}\sup_{\theta\in\Theta_{\delta}}\left| \sum_{j\geq 0}\gamma_{k+j}(\theta)X_{-j}\right|>\beta\right)\\
&\leq\mathbb{P}\left(\sum_{j\geq 0}\sup_{k\geq t}\sup_{\theta\in\Theta_{\delta}}\left| \gamma_{k+j}(\theta)\right|\left|X_{-j}\right|>\beta\right)\\
&\leq \frac{K}{\beta}\left( \sup_{t\in\mathbb{Z}}\mathbb{E} \left|X_t\right|\right )\sum_{j\geq 0}\left( \frac{1}{t+j}\right)^{1+d_1}
\\ &\leq \frac{K\mathrm{Var}(X_1)}{\beta d_1}(t-1)^{-d_1}\underset{t\to\infty}{\longrightarrow}0,
\end{align*}
which completes the proof of the convergence in \eqref{conv:ecart-eps}.
\end{proof}
\begin{rmq}\label{rmq:ecart-deriv} Since, for large $j$, $\partial\gamma_j(\theta)/\partial\theta_{k_1}=\mathrm{O}(j^{-1-d}\log(j))$ and $\partial^2\gamma_j(\theta)/\partial\theta_{k_1}\theta_{k_2}=\mathrm{O}(j^{-1-d}\{\log(j)\}^2)$, this last proposition remains valid for the first and second derivatives of $\epsilon_t(\theta)$. Following the same arguments developed in the proof of Proposition~\ref{prop:ecart-eps}, we have, almost-surely and for any $i,j\in\{1,\ldots,p+q+1\}$,
\begin{equation*}
\lim_{t\rightarrow \infty} \sup_{\theta\in\Theta_{\delta}}\left| \frac{\partial\epsilon_t(\theta)}{\partial\theta_i}-\frac{\partial\tilde{\epsilon}_t(\theta)}{\partial\theta_i}\right|=0
\end{equation*}
and
\begin{equation*}
\lim_{t\rightarrow \infty} \sup_{\theta\in\Theta_{\delta}}\left| \frac{\partial^2\epsilon_t(\theta)}{\partial\theta_i\partial\theta_j}-\frac{\partial^2\tilde{\epsilon}_t(\theta)}{\partial\theta_i\partial\theta_j}\right|=0.
\end{equation*}
\end{rmq}

Since our assumptions are made on the noise in $(\ref{FARIMA})$, it will be useful to express the random variables $\epsilon_t(\theta)$ and its partial derivatives with respect to $\theta$, as a function of $(\epsilon_{t-i})_{i\ge 0}$.

From \eqref{epsi-bis}, there exists a sequence $\lambda(\theta)=(\lambda_i(\theta))_{i\geq0}$ such that
\begin{equation}\label{epsil-th}
\epsilon_t(\theta)=\sum_{i=0}^{\infty}\lambda_i\left( \theta\right) \epsilon_{t-i}
\end{equation}
where $\lambda(\theta)$ is given by the sequence of the coefficients of the power series $b^{-1}_{\theta}(z) a_{\theta}(z)(1-z)^{d-d_0}a^{-1}_{\theta_0}(z) b_{\theta_0}(z)$. Consequently $\lambda(\theta) = \gamma(\theta)\ast \eta(\theta_0)$ or, equivalently,
\begin{align}\label{Coef-lambda}
\lambda_i( \theta)& =\sum_{j=0}^i\gamma_j(\theta)\eta_{i-j}(\theta_0).
\end{align}
As in \cite{hualde2011}, it can be shown using Stirling's approximation that there exists a positive constant $K$ such that 
\begin{equation}\label{eqasym-lambda}
\sup_{\theta\in\Theta_{\delta}}\left| \lambda_i(\theta)\right| \leq K\sup_{d\in[d_1,d_2]} i^{-1-(d-d_0)}\leq K i^{-1-(d_1-d_0)} \ .
\end{equation}

\noindent Equation \eqref{epsil-th} and Inequality \eqref{eqasym-lambda} imply that for all $\theta\in\Theta$ the random variable $\epsilon_t(\theta)$ belongs to $\mathbb{L}^2$, that 
the sequence $(\epsilon_t(\theta))_{t\in\mathbb{Z}}$ is an ergodic sequence and that for all $t\in\mathbb{Z}$ the function $\epsilon_t(\cdot)$ is a continuous function.
We proceed in the same way as regard to the derivatives of $\epsilon_t(\theta)$. More precisely, for any $\theta\in\Theta$, $t\in\mathbb{Z}$ and $1\le k,l \le p+q+1$ there exists sequences $\overset{\textbf{.}}{\lambda}_{k}(\theta)= (\overset{\textbf{.}}{\lambda}_{i,k}(\theta))_{i\geq1}$
and $\overset{\textbf{..}}{\lambda}_{k,l}(\theta)= (\overset{\textbf{..}}{\lambda}_{i,k,l}(\theta))_{i\geq1}$ such that
\begin{align}
\frac{\partial\epsilon_t(\theta)}{\partial\theta_k}&=\sum_{i=1}^{\infty}\overset{\textbf{.}}{\lambda}_{i,k}\left( \theta\right) \epsilon_{t-i} 
 \label{deriveesecepsil}\\
\frac{\partial^2\epsilon_t(\theta)}{\partial\theta_k\partial\theta_{l}}& =\sum_{i=1}^{\infty}\overset{\textbf{..}}{\lambda}_{i,k,l}\left( \theta\right) \epsilon_{t-i} .\label{deriveesecepsil1}
\end{align}
Of course it holds that $\overset{\textbf{.}}{\lambda}_{k}(\theta)=\frac{\partial\gamma(\theta)}{\partial\theta_k}\ast\eta(\theta_0)$ and
$\overset{\textbf{..}}{\lambda}_{k,l}( \theta)=\frac{\partial^2\gamma(\theta)}{\partial\theta_k\partial\theta_{l}}\ast \eta(\theta_0)$.

Similarly, we have
\begin{align}\label{epsiltilde-th}
\tilde{\epsilon}_t(\theta)& =\sum_{i=0}^{\infty}\lambda_i^t\left( \theta\right) \epsilon_{t-i}, \\
\frac{\partial\tilde{\epsilon}_t(\theta)}{\partial\theta_k}& =\sum_{i=1}^{\infty}\overset{\textbf{.}}{\lambda}_{i,k}^t\left( \theta\right) \epsilon_{t-i},
\label{deriveesecepsiltilde} \\
\frac{\partial^2\tilde{\epsilon}_t(\theta)}{\partial\theta_k\partial\theta_{l}}& =\sum_{i=1}^{\infty}\overset{\textbf{..}}{\lambda}_{i,k,l}^t\left( \theta\right) \epsilon_{t-i},
\end{align}
where $\lambda^t(\theta) = \gamma^t(\theta)\ast \eta(\theta_0)$, $\overset{\textbf{.}}{\lambda}^t_{k}(\theta)=\frac{\partial\gamma^t(\theta)}{\partial\theta_k}\ast\eta(\theta_0)$ and
$\overset{\textbf{..}}{\lambda}^t_{k,l}( \theta)=\frac{\partial^2\gamma^t(\theta)}{\partial\theta_k\partial\theta_{l}}\ast \eta(\theta_0)$.

\subsection{Proof of Theorem~\ref{th:convergence}}
We use \eqref{eq: one-step est} and a Taylor expansion of the function $\partial Q_n(\cdot)/ \partial\theta$ around $\theta_0$ to obtain 
\begin{align}\label{eq:tay}
\overline{\theta}_n-\theta_0&=\left(\theta^{*}_n-\theta_0\right)-\left\lbrace\frac{\partial^2}{\partial\theta\partial\theta^\prime}Q_n\left(\theta^{*}_n\right)\right\rbrace^{-1}\left\lbrace\frac{\partial}{\partial\theta}Q_n\left(\theta_0\right)+\left[\frac{\partial^2}{\partial\theta_i\partial\theta_j}Q_n\left(\tilde{\theta}_{n,i,j}\right)\right]\left(\theta^{*}_n-\theta_0\right)\right\rbrace,
\end{align}
where the $\tilde{\theta}_{n,i,j}$'s are between $\theta_n^{*}$ and $\theta_0$. In the two following lemmas, we show respectively the almost-sure convergence of $\partial^2 Q_n(\theta^{*}_n)/\partial\theta\partial\theta^\prime$ to $J(\theta_0)$ and that of $\partial Q_n(\theta_0)/\partial\theta$ to 0.
\begin{lemme} \label{lemme:der2Qn}Under the assumptions of Theorem~\ref{th:convergence}, we have almost-surely
\begin{equation*}
\lim_{n\to\infty}\frac{\partial^2}{\partial\theta\partial\theta^\prime}Q_n\left(\theta^{*}_n\right)=J(\theta_0).
\end{equation*}
\end{lemme}
\begin{proof}
For any $\theta\in\Theta$, let
\begin{align*}
{J}_n(\theta)&=\frac{\partial^2}{\partial\theta\partial\theta^{'}}Q_n\left( \theta\right)=\frac{2}{n}\sum_{t=1}^n\left\lbrace \frac{\partial}{\partial\theta}\tilde{\epsilon}_t\left(\theta\right)\right\rbrace \left\lbrace \frac{\partial}{\partial\theta^{'}}\tilde{\epsilon}_t\left(\theta\right)\right\rbrace+\frac{2}{n}\sum_{t=1}^n\tilde{\epsilon}_t(\theta)\frac{\partial^2}{\partial\theta\partial\theta^{'}}\tilde{\epsilon}_t(\theta),
\end{align*}
and
\begin{align*}
J_n^{*}(\theta)&=\frac{\partial^2}{\partial\theta\partial\theta^{'}}O_n\left( \theta\right)=\frac{2}{n}\sum_{t=1}^n\left\lbrace \frac{\partial}{\partial\theta}\epsilon_t\left(\theta\right)\right\rbrace \left\lbrace \frac{\partial}{\partial\theta^{'}}\epsilon_t\left(\theta\right)\right\rbrace+\frac{2}{n}\sum_{t=1}^n\epsilon_t(\theta)\frac{\partial^2}{\partial\theta\partial\theta^{'}}\epsilon_t(\theta).
\end{align*}
It is clear that for any $i,j\in\{1,\ldots,p+q+1\}$,
\begin{align}\label{ineq:J}
\left|\frac{\partial^2}{\partial\theta_i\partial\theta_j}Q_n\left( \theta^*_n\right)-J(\theta_0)(i,j)\right|&\leq \left|{J}_{n}(\theta^*_n)(i,j)-J_{n}^{*}(\theta^*_{n})(i,j)\right|\nonumber\\
&\quad+\left|J_{n}^{*}(\theta^*_n)(i,j)-J_{n}^{*}(\theta_0)(i,j)\right|+\left|J_{n}^{*}(\theta_0)(i,j)-J(\theta_0)(i,j)\right|.
\end{align}
So it is enough to show that the three terms in the right hand side of \eqref{ineq:J} converge almost-surely to $0$ when $n$ tends to infinity.
The random variable $\epsilon_t$ is uncorrelated with ${\partial\epsilon_t(\theta_0)}/{\partial\theta}$ and ${\partial^2\epsilon_t(\theta_0)}/{\partial\theta\partial\theta}$ (this is due to \eqref{deriveesecepsil} and \eqref{deriveesecepsil1} and the non correlation of the innovation process $(\epsilon_t)_{t\in\mathbb{Z}}$). Thus, by the ergodicity of process $(\epsilon_t)_{t\in\mathbb{Z}}$ assumed in \textbf{(A1)}, we have
\begin{equation*}
J_n^{*}(\theta_0)\underset{n\to\infty}{\overset{\text{a.s.}}{\longrightarrow}}2\mathbb{E}\left[ \frac{\partial}{\partial\theta}\epsilon_t(\theta_0)\frac{\partial}{\partial\theta^{'}}\epsilon_t(\theta_0)\right]=J(\theta_0).
\end{equation*}
Let us now show that the term $|J_{n}^{*}(\theta^*_{n})(i,j)-J_{n}^{*}(\theta_0)(i,j)|$ converges almost-surely to 0.

In view of  \eqref{AR-Inf} and \eqref{Coef-Gamma}, one successively has
\begin{align*}
\sup_{\theta\in\Theta_{\delta}}&\left\|\frac{\partial}{\partial\theta}\left(\frac{\partial}{\partial\theta_i}\epsilon_t(\theta)\frac{\partial}{\partial\theta_j}\epsilon_t(\theta) \right)\right\|\\
&=\sup_{\theta\in\Theta_{\delta}}\left\|\frac{\partial}{\partial\theta}\left\lbrace \left(\sum_{k_1\geq 1}\frac{\partial}{\partial\theta_i}\gamma_{k_1}(\theta)X_{t-k_1}\right) \left( \sum_{k_2\geq 1}\frac{\partial}{\partial\theta_j}\gamma_{k_2}(\theta)X_{t-k_2}\right)\right\rbrace \right\|\\
&=\sup_{\theta\in\Theta_{\delta}}\left\|\frac{\partial}{\partial\theta}\left(\sum_{k_1,k_2\geq 1}\frac{\partial}{\partial\theta_i}\gamma_{k_1}(\theta)\frac{\partial}{\partial\theta_j}\gamma_{k_2}(\theta)X_{t-k_1}X_{t-k_2}\right)\right\|\\
&\leq \sup_{\theta\in\Theta_{\delta}}\left\|\sum_{k_1,k_2\geq 1}\left(\frac{\partial}{\partial\theta}\frac{\partial}{\partial\theta_i}\gamma_{k_1}(\theta)\right)\frac{\partial}{\partial\theta_j}\gamma_{k_2}(\theta)X_{t-k_1}X_{t-k_2}\right\|\\
&\hspace{1cm}+\sup_{\theta\in\Theta_{\delta}}\left\|\sum_{k_1,k_2\geq 1}\frac{\partial}{\partial\theta_i}\gamma_{k_1}(\theta)\left(\frac{\partial}{\partial\theta}\frac{\partial}{\partial\theta_j}\gamma_{k_2}(\theta)\right)X_{t-k_1}X_{t-k_2}\right\|\\
&\leq \sum_{k_1,k_2\geq 1}\sup_{\theta\in\Theta_{\delta}}\left\|\frac{\partial}{\partial\theta}\frac{\partial}{\partial\theta_i}\gamma_{k_1}(\theta)\right\|\sup_{\theta\in\Theta_{\delta}}\left\|\frac{\partial}{\partial\theta_j}\gamma_{k_2}(\theta)\right\|\left|X_{t-k_1}\right|\left|X_{t-k_2}\right|\\
&\hspace{1cm}+\sum_{k_1,k_2\geq 1}\sup_{\theta\in\Theta_{\delta}}\left\|\frac{\partial}{\partial\theta_i}\gamma_{k_1}(\theta)\right\|\sup_{\theta\in\Theta_{\delta}}\left\|\frac{\partial}{\partial\theta}\frac{\partial}{\partial\theta_j}\gamma_{k_2}(\theta)\right\|\left|X_{t-k_1}\right|\left|X_{t-k_2}\right|\\
&\leq K\sum_{k_1,k_2\geq 1}\left(\log(k_1) \right)^2k_1^{-1-d_1}\log(k_2)k_2^{-1-d_1} \left|X_{t-k_1}\right|\left|X_{t-k_2}\right|\\
&\hspace{1cm}+K\sum_{k_1,k_2\geq 1}\log(k_1)k_1^{-1-d_1}\left(\log(k_2) \right)^2k_2^{-1-d_1} \left|X_{t-k_1}\right|\left|X_{t-k_2}\right| .
\end{align*}
Consequently, we obtain
\begin{align}\nonumber
\mathbb{E}_{\theta_0}\left[ \sup_{\theta\in\Theta_{\delta}}\left\|\frac{\partial}{\partial\theta}\left(\frac{\partial}{\partial\theta_i}\epsilon_t(\theta)\frac{\partial}{\partial\theta_j}\epsilon_t(\theta) \right)\right\|\right]&\leq K\sum_{k_1,k_2\geq 1}\left(\log(k_1) \right)^2k_1^{-1-d_1}\log(k_2)k_2^{-1-d_1} \sup_{t\in\mathbb{Z}}\mathbb{E}_{\theta_0}\left|X_t\right|^2\\ \nonumber
&\hspace{0.5cm}+K\sum_{k_1,k_2\geq 1}\log(k_1)k_1^{-1-d_1}\left(\log(k_2) \right)^2k_2^{-1-d_1} \sup_{t\in\mathbb{Z}}\mathbb{E}_{\theta_0}\left|X_t\right|^2 \\
&\le K. \label{derivee_troisieme}
\end{align}
Following the same approach used in \eqref{derivee_troisieme}, we have
\begin{align}\label{third_derivatives}
\mathbb{E}_{\theta_0}\left[ \sup_{\theta\in\Theta_{\delta}}\left\|\frac{\partial}{\partial\theta}\left\lbrace\epsilon_t(\theta)\frac{\partial^2}{\partial\theta_i\partial\theta_j}\epsilon_t(\theta) \right\rbrace\right\|\right]<\infty.
\end{align}
A Taylor expansion implies that there exists a random variable $\theta_{n}^{**}$ between $\theta^*_{n}$ and $\theta_0$ such that
\begin{align*}
\left|J_n^{*}(\theta^*_{n})(i,j)-J_n^{*}(\theta_0)(i,j)\right|&=\left|\frac{\partial}{\partial\theta}J_n^{*}(\theta_{n}^{**})(i,j)\cdot(\theta^*_{n}-\theta_0)\right| \\
& \leq \sup_{\theta\in\Theta_{\delta}}\left\|\frac{\partial}{\partial\theta}J_n^{*}(\theta)(i,j)\right\|\left\|\theta^*_{n}-\theta_0\right\|\\
&\leq \frac{2}{n}\sum_{t=1}^n\sup_{\theta\in\Theta_{\delta}}\left\|\frac{\partial}{\partial\theta}\left(\frac{\partial}{\partial\theta_i}\epsilon_t(\theta)\frac{\partial}{\partial\theta_j}\epsilon_t(\theta) \right)\right\|\left\|\theta^*_{n}-\theta_0\right\|\\
&\hspace{1cm}+\frac{2}{n}\sum_{t=1}^n\sup_{\theta\in\Theta_{\delta}}\left\|\frac{\partial}{\partial\theta}\left\lbrace\epsilon_t(\theta)\frac{\partial^2}{\partial\theta_i\partial\theta_j}\epsilon_t(\theta) \right\rbrace\right\|\left\|\theta^*_{n}-\theta_0\right\| .
\end{align*}
Proposition~\ref{prop:ecart-eps} (which implies the almost-sure convergence of the LSE $\theta^*_{n}$ to $\theta_0$), the ergodic theorem and Equations \eqref{derivee_troisieme} and \eqref{third_derivatives} imply that $\lim_{n\to\infty} |J_n^{*}(\theta^*_{n})(i,j)-J_n^{*}(\theta_0)(i,j)|=0$ a.s.\\

To prove the almost-sure convergence of the first term of the right hand side of \eqref{ineq:J} it suffices to show that
$$\frac{1}{n}\sum_{t=1}^n\sup_{\theta\in\Theta_{\delta}}\left|\frac{\partial}{\partial\theta_i}\epsilon_t(\theta)\frac{\partial}{\partial\theta_j}\epsilon_t(\theta)-\frac{\partial}{\partial\theta_i}\tilde{\epsilon}_t(\theta)\frac{\partial}{\partial\theta_j}\tilde{\epsilon}_t(\theta)\right|$$
\noindent and
$$\frac{1}{n}\sum_{t=1}^n\sup_{\theta\in\Theta_{\delta}}\left|\epsilon_t(\theta)\frac{\partial^2}{\partial\theta_i\partial\theta_j}\epsilon_t(\theta)-\tilde{\epsilon}_t(\theta)\frac{\partial^2}{\partial\theta_i\partial\theta_j}\tilde{\epsilon}_t(\theta)\right|$$
converge  almost-surely to 0. On one hand, we have
\begin{align*}
\frac{1}{n}\sum_{t=1}^n&\sup_{\theta\in\Theta_{\delta}}\left|\frac{\partial}{\partial\theta_i}\epsilon_t(\theta)\frac{\partial}{\partial\theta_j}\epsilon_t(\theta)-\frac{\partial}{\partial\theta_i}\tilde{\epsilon}_t(\theta)\frac{\partial}{\partial\theta_j}\tilde{\epsilon}_t(\theta)\right|\\
&\leq \frac{1}{n}\sum_{t=1}^n\left\lbrace  \sup_{\theta\in\Theta_{\delta}}\left|\frac{\partial}{\partial\theta_i}\epsilon_t(\theta)-\frac{\partial}{\partial\theta_i}\tilde{\epsilon}_t(\theta)\right|\sup_{\theta\in\Theta_{\delta}}\left|\frac{\partial}{\partial\theta_j}\epsilon_t(\theta)\right|\right.\\
&\qquad\left.+\sup_{\theta\in\Theta_{\delta}}\left|\frac{\partial}{\partial\theta_i}\tilde{\epsilon}_t(\theta)\right|\sup_{\theta\in\Theta_{\delta}}\left|\frac{\partial}{\partial\theta_j}\tilde{\epsilon}_t(\theta)-\frac{\partial}{\partial\theta_j}\epsilon_t(\theta)\right|\right\rbrace \\
&\leq \left (\frac{1}{n}\sum_{t=1}^n\left( \sup_{\theta\in\Theta_{\delta}}\left|\frac{\partial}{\partial\theta_i}\epsilon_t(\theta)-\frac{\partial}{\partial\theta_i}\tilde{\epsilon}_t(\theta)\right|\right) ^2\right )^{1/2}\left (\frac{1}{n}\sum_{t=1}^n\left( \sup_{\theta\in\Theta_{\delta}}\left|\frac{\partial}{\partial\theta_j}\epsilon_t(\theta)\right|\right) ^2\right )^{1/2}\\
&\qquad\qquad+\left (\frac{1}{n}\sum_{t=1}^n\left( \sup_{\theta\in\Theta_{\delta}}\left|\frac{\partial}{\partial\theta_i}\tilde{\epsilon}_t(\theta)\right|\right) ^2\right )^{1/2}\left (\frac{1}{n}\sum_{t=1}^n\left( \sup_{\theta\in\Theta_{\delta}}\left|\frac{\partial}{\partial\theta_j}\tilde{\epsilon}_t(\theta)-\frac{\partial}{\partial\theta_j}\epsilon_t(\theta)\right|\right) ^2\right )^{1/2}.
\end{align*}
From \eqref{AR-Inf} and \eqref{Coef-Gamma}, it follows that
\begin{align*}
\mathbb{E}_{\theta_0}\left[\left(\sup_{\theta\in\Theta_{\delta}}\left|\frac{\partial}{\partial\theta_j}\epsilon_t(\theta)\right|\right) ^2\right]
&\leq \sup_{t\in\mathbb{Z}}\mathbb{E}_{\theta_0}\left|X_t\right|^2\left(\sum_{k_1\geq 1}\log(k_1)k_1^{-1-d_1} \right) ^2
<\infty.
\end{align*}
Similar calculations can be done to obtain
$$\mathbb{E}_{\theta_0}\left[\left(\sup_{\theta\in\Theta_{\delta}}\left|\frac{\partial}{\partial\theta_i}\tilde{\epsilon}_t(\theta)\right|\right) ^2\right] <\infty.$$
Ces\`{a}ro's Lemma, Remark~\ref{rmq:ecart-deriv} and the ergodic theorem yield
\begin{align*}
\frac{1}{n}\sum_{t=1}^n\sup_{\theta\in\Theta_{\delta}}\left|\frac{\partial}{\partial\theta_i}\epsilon_t(\theta)\frac{\partial}{\partial\theta_j}\epsilon_t(\theta)-\frac{\partial}{\partial\theta_i}\tilde{\epsilon}_t(\theta)\frac{\partial}{\partial\theta_j}\tilde{\epsilon}_t(\theta)\right|\underset{n\to\infty}{\overset{\text{a.s.}}{\longrightarrow}}0.
\end{align*}
On the other hand, one similarly may prove that
$$\frac{1}{n}\sum_{t=1}^n\sup_{\theta\in\Theta_{\delta}}\left|\epsilon_t(\theta)\frac{\partial^2}{\partial\theta_i\partial\theta_j}\epsilon_t(\theta)-\tilde{\epsilon}_t(\theta)\frac{\partial^2}{\partial\theta_i\partial\theta_j}\tilde{\epsilon}_t(\theta)\right|\underset{n\to\infty}{\overset{\text{a.s.}}{\longrightarrow}}0.$$
Thus
\begin{align*}
\sup_{\theta\in\Theta_{\delta}}\left\|{J}_n(\theta)-J_n^{*}(\theta)\right\|
\underset{n\to\infty}{\overset{\text{a.s.}}{\longrightarrow}}0
\end{align*}
and the lemma is proved.
\end{proof}
\begin{lemme}\label{lemme:derQn} Under the assumptions of Theorem~\ref{th:convergence}, we have almost-surely
\begin{equation*}
\lim_{n\to\infty}\frac{\partial}{\partial\theta}Q_n\left(\theta_0\right)=0.
\end{equation*}
\end{lemme}
\begin{proof}
Observe first that, for any $i\in\{1,\ldots,p+q+1\}$,
\begin{equation}\label{eq:der-Qn}
\frac{\partial}{\partial\theta_i}Q_n\left(\theta_0\right)=\frac{2}{n}\sum_{t=1}^n\left(\tilde{\epsilon}_t(\theta_0)\frac{\partial\tilde{\epsilon}_t(\theta_0)}{\partial\theta_i}-\epsilon_t\frac{\partial\epsilon_t(\theta_0)}{\partial\theta_i}\right)+\frac{2}{n}\sum_{t=1}^n\epsilon_t\frac{\partial\epsilon_t(\theta_0)}{\partial\theta_i}.
\end{equation}
Since $\epsilon_t$ and $\partial\epsilon_t(\theta_0)/\partial\theta_i$ are uncorrelated, the ergodic theorem in \textbf{(A1)} implies that the second term in the right hand side of \eqref{eq:der-Qn} converges almost-surely to 0 as $n\to\infty$.\\
Note now that 
\begin{align*}
\tilde{\epsilon}_t(\theta_0)\frac{\partial\tilde{\epsilon}_t(\theta_0)}{\partial\theta_i}-\epsilon_t\frac{\partial\epsilon_t(\theta_0)}{\partial\theta_i}&=\left(\tilde{\epsilon}_t(\theta_0)-\epsilon_t\right)\frac{\partial\tilde{\epsilon}_t(\theta_0)}{\partial\theta_i}+\epsilon_t\left(\frac{\partial\tilde{\epsilon}_t(\theta_0)}{\partial\theta_i}-\frac{\partial\epsilon_t(\theta_0)}{\partial\theta_i}\right).
\end{align*}
We use, as in the proof of Lemma \ref{lemme:der2Qn}, Proposition~\ref{prop:ecart-eps} and Remark~\ref{rmq:ecart-deriv} to complete the proof of the lemma.
\end{proof}

The proof of the theorem follows from Proposition~\ref{prop:ecart-eps}, Lemma 16 of \citep{BMES2019} and Lemmas \ref{lemme:der2Qn} and \ref{lemme:derQn}.

\subsection{Proof of Theorem~\ref{th:asy-norm}}

In view of \eqref{eq: one-step est} and by a Taylor expansion of the function $\partial Q_n(\cdot)/ \partial\theta$ around $\theta_0$, we have
\begin{align*}
\sqrt{n}\left(\overline{\theta}_n-\theta_0\right)&=\sqrt{n}\left(\theta^{*}_n-\theta_0\right)-\sqrt{n}\left\lbrace\frac{\partial^2}{\partial\theta\partial\theta^\prime}Q_n\left(\theta^{*}_n\right)\right\rbrace^{-1}\left\lbrace\frac{\partial}{\partial\theta}Q_n\left(\theta_0\right)+\left[\frac{\partial^2}{\partial\theta_i\partial\theta_j}Q_n\left(\tilde{\theta}_{n,i,j}\right)\right]\left(\theta^{*}_n-\theta_0\right)\right\rbrace,
\end{align*}
where the $\tilde{\theta}_{n,i,j}$'s are between $\theta_n^{*}$ and $\theta_0$.
\noindent Hence, it follows that
\begin{align*}
\sqrt{n}\left(\overline{\theta}_n-\theta_0\right)&=\left\lbrace\frac{\partial^2}{\partial\theta\partial\theta^\prime}Q_n\left(\theta^{*}_n\right)\right\rbrace^{-1}n^{\delta/2}\left\lbrace\frac{\partial^2}{\partial\theta\partial\theta^\prime}Q_n\left(\theta^{*}_n\right)-\left[\frac{\partial^2}{\partial\theta_i\partial\theta_j}Q_n\left(\tilde{\theta}_{n,i,j}\right)\right]\right\rbrace n^{\delta/2}\left(\theta^{*}_n-\theta_0\right)n^{1/2-\delta}\\
&\qquad -\left\lbrace\frac{\partial^2}{\partial\theta\partial\theta^\prime}Q_n\left(\theta^{*}_n\right)\right\rbrace^{-1}\sqrt{n}\frac{\partial}{\partial\theta}Q_n\left(\theta_0\right).
\end{align*}
We use Lemmas 16-19 of \cite{BMES2019}, Proposition~\ref{prop:lipschitz} and Slutsky's theorem to complete the proof.

\subsection{Proof of Proposition~\ref{prop:lipschitz}}

For any $\theta^{(1)}, \theta^{(2)}\in\Theta$, the mean value theorem gives 
\begin{equation*}
\frac{\partial^2}{\partial\theta_i\partial\theta_j}Q_n\left(\theta^{(1)}\right)-\frac{\partial^2}{\partial\theta_i\partial\theta_j}Q_n\left(\theta^{(2)}\right)=\frac{\partial^3}{\partial\theta\partial\theta_i\partial\theta_j}Q_n\left((1-c)\theta^{(1)}+c\theta^{(2)}\right)\mathbf{\cdot} \left\lbrace \theta^{(1)}-\theta^{(2)}\right\rbrace,
\end{equation*}
for some $c$ between 0 and 1. \\
In view of \eqref{Qn} and for all $\theta\in\Theta$, a simple calculation of derivative leads to
\begin{equation*}
\frac{\partial^3}{\partial\theta_i\partial\theta_j\partial\theta_k}Q_n\left(\theta\right)=T_{1,n}(\theta)+T_{2,n}(\theta)+T_{3,n}(\theta)+T_{4,n}(\theta),
\end{equation*}
where
\begin{align*}
T_{1,n}(\theta)&=\frac{2}{n}\sum_{t=1}^n\frac{\partial^2\tilde{\epsilon}_t(\theta)}{\partial\theta_i\partial\theta_j}\frac{\partial\tilde{\epsilon}_t(\theta)}{\partial\theta_k}, \\
T_{2,n}(\theta)&=\frac{2}{n}\sum_{t=1}^n\frac{\partial\tilde{\epsilon}_t(\theta)}{\partial\theta_j}\frac{\partial^2\tilde{\epsilon}_t(\theta)}{\partial\theta_i\partial\theta_k},\\ 
T_{3,n}(\theta)&=\frac{2}{n}\sum_{t=1}^n\frac{\partial\tilde{\epsilon}_t(\theta)}{\partial\theta_i}\frac{\partial^2\tilde{\epsilon}_t(\theta)}{\partial\theta_j\partial\theta_k}, 
\end{align*}
and 
\begin{align*}
T_{4,n}(\theta)&=\frac{2}{n}\sum_{t=1}^n\tilde{\epsilon}_t(\theta)\frac{\partial^3\tilde{\epsilon}_t(\theta)}{\partial\theta_i\partial\theta_j\partial\theta_k}.
\end{align*}
We use Equations \eqref{epsilon-tilde--} and \eqref{Coef-Gamma} to obtain
\begin{align*}
\mathbb{E}\left[\sup_{\theta\in\Theta_\kappa}\left|T_{1,n}(\theta)\right|\right]&\leq \frac{2}{n}\sum_{t=1}^n\mathbb{E}\left[\sup_{\theta\in\Theta_\kappa}\left|\frac{\partial^2\tilde{\epsilon}_t(\theta)}{\partial\theta_i\partial\theta_j}\right|\sup_{\theta\in\Theta_\kappa}\left|\frac{\partial\tilde{\epsilon}_t(\theta)}{\partial\theta_k}\right|\right]\\
&=\frac{2}{n}\sum_{t=1}^n\mathbb{E}\left[\sup_{\theta\in\Theta_\kappa}\left|\sum_{\ell=1}^\infty \frac{\partial^2\gamma_{\ell}^t(\theta)}{\partial\theta_i\partial\theta_j}X_{t-\ell}\right|\sup_{\theta\in\Theta_\kappa}\left|\sum_{\ell=1}^\infty \frac{\partial\gamma_{\ell}^t(\theta)}{\partial\theta_k}X_{t-\ell}\right|\right]\\
&\leq\frac{2}{n}\sum_{t=1}^n\sum_{\ell_1,\ell_2=1}^{t-1}\sup_{\theta\in\Theta_\kappa}\left| \frac{\partial^2\gamma_{\ell_1}(\theta)}{\partial\theta_i\partial\theta_j}\right|\sup_{\theta\in\Theta_\kappa}\left| \frac{\partial\gamma_{\ell_2}(\theta)}{\partial\theta_k}\right|\mathbb{E}\left[\left| X_{t-\ell_1}X_{t-\ell_2}\right|\right]\\
&\leq 2\sup_{t\in\mathbb{Z}}\mathbb{E} \left[X_t^2\right]\left(\sum_{\ell=1}^{\infty}\ell^{-1-d_1}\left\lbrace\log(\ell)\right\rbrace^2\right)\left(\sum_{\ell=1}^{\infty}\ell^{-1-d_1}\log(\ell)\right)<\infty.
\end{align*}
Thanks to Markov's inequality, we deduce that 
$$\sup_{\theta\in\Theta_\kappa}\left|T_{1,n}(\theta)\right|=\mathrm{O}_{\mathbb{P}}(1).$$
Similar calculation can be done to show that $T_{2,n}(\theta)$, $T_{3,n}(\theta)$ and $T_{4,n}(\theta)$ are bounded in probability uniformly in $\theta$. It follows then that 
$$\sup_{\theta\in\Theta_\kappa}\left\|\frac{\partial^3}{\partial\theta\partial\theta_i\partial\theta_j}Q_n\left(\theta\right)\right\|=\mathrm{O}_{\mathbb{P}}(1).$$
This is enough to complete the proof.\\

\textbf{Acknowledgments.} This research benefited from the support of the ANR project "Efficient inference for large and high-frequency data" (ANR-21-CE40-0021), the "Chair Risques \'Emergents ou Atypiques en Assurance", under the aegis of Fondation du Risque, a joint initiative by Le Mans University and MMA company, member of Covea group and the "Chair Impact de la Transition Climatique en Assurance", under the aegis of Fondation du Risque, a joint initiative by Le Mans University and Groupama Centre-Manche company, member of Groupama group.

\newpage

\bibliographystyle{apalike}
\bibliography{biblio-youss}
\end{document}